\documentclass[final]{siamart190516}
\usepackage[T1]{fontenc}
\usepackage[utf8]{inputenc}
\usepackage{amsmath}
\usepackage{amssymb}
\usepackage{graphicx,booktabs}
\usepackage{times}
\usepackage[notcite,notref]{showkeys}
\usepackage{enumitem}

\usepackage{algpseudocode}

\Crefname{ALC@unique}{Step}{Lines} 

\usepackage[utf8]{inputenc}
\usepackage[english]{babel}
\usepackage{xcolor}
\def\BA{{\bf A}}
\def\BAT{{\bf A}^{\!\!\top}}
\def\RR{\mathbb{R}}
\def\S{\mathbb{S}}

\def\tn{\mathop{{\widetilde{n}}}}

\DeclareMathOperator{\rank}{rank}

\newtheorem{Assumption}{Assumption}[section]
\newtheorem{Example}{Example}[section]
\newtheorem*{remark}{Remark}

\title{Barrier and penalty methods for low-rank semidefinite programming with application to truss topology design\thanks{{This work has been supported by European Union’s Horizon 2020 research and innovation programme under the Marie Skłodowska-Curie Actions, grant agreement 813211 (POEMA)}}}

\author{Soodeh Habibi\thanks{School of Mathematics, University of Birmingham, United Kingdom} \and Arefeh Kavand\thanks{Applied Mathematics (Continuous Optimization), Friedrich-Alexander-Universität Erlangen-Nürnberg, Germany} \and Michal Ko\v{c}vara\thanks{School of Mathematics, University of Birmingham, United Kingdom and Institute of Information Theory and Automation, Prague, Czech Republic} \and Michael Stingl\thanks{Applied Mathematics (Continuous Optimization), Friedrich-Alexander-Universität Erlangen-Nürnberg, Germany}}

\date{\today}

\begin{document}

\maketitle
\begin{abstract}
    The aim of this paper is to solve large-and-sparse linear Semidefinite Programs (SDPs) with low-rank solutions. We propose to use a preconditioned conjugate gradient method within second-order SDP algorithms and introduce a new efficient preconditioner fully utilizing the low-rank information. We demonstrate that the preconditioner is universal, in the sense that it can be efficiently used within a standard interior-point algorithm, as well as a newly developed primal-dual penalty method. 
    The efficiency is demonstrated by numerical experiments using the truss topology optimization problems of growing dimension.
\end{abstract}

\section{Introduction}\label{sec:intro}

First efficient solvers for semidefinite optimization emerged more than twenty years ago. All of the general purpose solvers	have been using second-order algorithms, predominantly the interior-point (IP) method. In every iteration of such an algorithm, one has to solve a system of linear equations 
\begin{equation*}\label{eq:eq}
  Hx = g
\end{equation*}
of size $n\times n$, where $n$ is the dimension of the unknown vector $x$ in the general semidefinite programming (SDP) problem
	\begin{equation*}
	\begin{aligned}
	& \min_{x\in\RR^{n}} c^\top x \quad
	 \mbox{subject to}
	\quad \sum_{i=1}^{{n}} x_i A_i^{(k)} - B^{(k)}\succeq 0\,,\quad k=1,\ldots,{p}
	\end{aligned}
	\end{equation*}
	with
	$A_i^{(k)}\!,\; B^{(k)}\in\RR^{{m}\times{m}}
	$.
The assembly and solution of this linear system is a well-known bottleneck of these solvers. This is the case even for problems with sparse data and sparse linear systems, when the assembly of the system requires $\mathcal{O} \left(pnm^3\right)$ flops, and the solution by sparse Cholesky factorization $\mathcal{O}\left(n^\alpha\right)$ flops with some $\alpha\in[1,3]$.

On the other hand, many real-world applications lead to large-scale SDP, often unsolvable by current general-purpose software. There are, essentially, three ways how to approach this conundrum:
\begin{enumerate}
	\item Reformulating the problem to make it more suitable for general-purpose solvers. This can be done, for instance, by facial reduction \cite{krislock2010explicit,permenter2018partial,zhu2019sieve} or by decomposition of large matrix inequalities (and thus reducing $m$) into several smaller ones \cite{kim2011exploiting,kovcvara2020decomposition}.
	\item Using a different algorithm, such as spectral bundle \cite{helmberg2000spectral}, ADMM \cite{oliveira2018admm}, augmented Lagrangian \cite{malick2009regularization,zhao2010newton}, optimization on manifolds~\cite{Journee}, randomized algorithms \cite{udell} or techniques of nonlinear programming \cite{Burer,burer-monteiro}.
	\item Using one of the second-order algorithms with Cholesky factorization replaced by an iterative method for the solution of the linear systems $Hx=g$; see, e.g., \cite{MKMS,toh2004solving,Zhang_2017}.
\end{enumerate}
The last approach is particularly attractive whenever $n\gg m$, and it addresses both bottlenecks of a second-order SDP solver:
\begin{itemize}
	\item The assembly of the system matrix; an iterative solver only needs matrix-vector multiplications, and so the matrix does not have to be explicitly assembled and stored.
	\item The solution of the linear system itself; an iterative solver can handle very large systems, as compared to a sparse Cholesky solver, under the assumption of good conditioning of the matrix or existence of a good preconditioner. 
\end{itemize}

Our goal in this paper is to design a preconditioner for the conjugate gradient (CG) method and use the resulting solver in two second-order SDP algorithms. The choice of an efficient preconditioner in a Krylov type method is problem dependent, in particular, in the context of optimization algorithms. We will focus on problems with expected very low-rank dual solutions. 

SDP problems with low-rank dual solutions are common in relaxations of optimization problems in different areas such as combinatorial optimization \cite{combinatorial}, approximation theory \cite{candes, lasserre}, control theory \cite{lasserre,zhangphdthesis}, structural optimization \cite{stingl2009sequential}, and power systems \cite{madani}. 

It is not our goal to find a low-rank dual solution in case of non-unique solutions; indeed, it is well-known that an interior-point method will converge to a maximal complementary solution. However, if the method will converge to a low-rank dual solution, we will use this fact to improve its computational complexity and memory demands. Our method is not only limited to problems with (theoretically exactly) low-rank solutions, it can be efficiently applied when the dual solution has a few outlying eigenvalues. Moreover, our method can be applied to problems with several linear matrix inequalities (LMIs) with different rank of the corresponding dual solutions and with bound constraints on the variable $x$, i.e., constraints that, mostly, do not exhibit low-rank duals.
	
There is one major difference to other algorithms for SDP problems with low-rank solutions, such as SDPLR \cite{Burer,burer-monteiro}, optimization on manifolds \cite{Journee} and other approaches \cite{bellavia,udell}. All these methods require the knowledge of a guaranteed (possibly tight) upper bound on the rank of the solution. If the estimate of the rank was lower than the actual rank, the algorithm would not find a solution. In our approach, the rank information is merely used to speed-up a standard, well-established algorithm. If our estimated rank is smaller than the actual one, we will only see it in possibly more iterations of the CG solver; the convergence behaviour of the optimization algorithm will be unchanged.

In this paper we present the application of the preconditioned iterative method in two second-order optimization algorithms, the ``standard'' primal-dual interior-point method and a new primal-dual augmented Lagrangian method (PDAL).
The interior-point algorithm mimics the reliable primal-dual algorithm with Nesterov-Todd direction, as used in the SDPT3 solver \cite{Todd-Toh,toh2004solving}.
%
	%
The augmented Lagrangian method is based on the PENNON algorithm in \cite{PENNON}, but implements a number of useful modifications. The most important change is to solve the subproblems in the augmented Lagrangian framework by a primal-dual approach inspired by the work of Griva and Polyak \cite{GrivaPolyak}. For this, we rephrase the Newton's system in the augmented Lagrangian algorithm by a primal-dual system which is then solved approximately by the preconditioned conjugate gradient (PCG) method. Additional modifications include the use of two different penalty functions for standard linear and LMI constraints---we consider a quadratic-logarithmic penalty function  \cite{YouzefoZibulevsky,BenTalZibulevsky,Zibulevsky,PHDThesisZibulevsky} for standard linear constraints and a hyperbolic penalty function \cite{PENNON,MKMS} to construct primary matrix function \cite{horn1991topics} for penalizing LMI constraints---as well as the application of a regularization term in the augmented Lagrangian function which bounds the eigenvalues of the Hessian matrix uniformly away from zero. 

Our preconditioner is motivated by the work of Zhang and Lavaei \cite{Zhang_2017}. They present a preconditioner for the CG method within a standard IP method that makes it more efficient for large-and-sparse low-rank SDPs. 
We partly follow their approach though there are some major differences: we abandon a crucial assumption on sparsity of the data matrices (Assumption 2 in \cite{Zhang_2017}); our method allows for
SDPs with more than one LMI and with (simple) linear constraints; the complexity of computing our preconditioner is substantially lower; we can also efficiently solve problems with a few outlying eigenvalues in the optimal solution.

To test the developed algorithms, we have generated a library of problems arising from truss topology optimization without and with vibration constraints, formulated as SDP. These problems are characterized by a very low-rank solutions (1--3), additional linear (bound) constraints and scalability, in the sense that one can generate a range of problems of various dimension but of the same type. Depending on the choice of the lower bound, the solutions are either of exact very low rank or an approximate one, with a very few outlying eigenvalues. We will demonstrate that, to solve the large-scale instances, both new elements of our engines have to be employed, the very use of an iterative solver (to address the matrix assembly and related memory issues) and an efficient preconditioner (to address the efficiency of the iterative solver). This is not always the case of the large scale SDP problems; for instance, the matrix completion problem requires the use of an iterative solver (within a second-order algorithm) but no preconditioner is needed in many instances.
	
The paper is organized as follows. 
Section~2 introduces the SDP problems we want to solve and some basic assumptions.
In Sections~3 and 4, PDAL and IP methods are presented. Also, in these sections, the complexity of the Schur complement system in IP and the Hessian system in PDAL 
are discussed. Next, in Section~5, we introduce the preconditioners. 
Section~6 describes our application, the truss topology optimization problems without and with vibration constraints.
In the last Section~7, we present two Matlab codes, PDAL and Loraine, implementing the developed algorithms, together with results of our numerical experiments.
We also give comparison with other existing SDP software and demonstrate the high efficiency of our algorithms.

\subsubsection*{Notation}
We denote by $\S^m,\S^m_+$ and $\S^m_{++}$, respectively, the space of $m\times m$ symmetric matrices, positive semidefinite and positive definite matrices.
	The eigenvalues of a given matrix $X\in \mathbb{S}_{++}^m$ are ordered as $\lambda_1(X)\leq\dots\leq \lambda_m(X)$, and its condition number is defined as $\kappa(X)=\lambda_1(X)/\lambda_n(X)$.
    The notation ``vec'' and ``$\otimes$" refer to the (non-symmetrized) vectorization and Kronecker product, respectively. 
    With this notation the following identity holds true: $\text{vec}(AXB^\top )=(A\otimes B)\text{vec}(X)$. 
    The symbol $\bullet$ denotes the Frobenius inner product of two matrices, $A \bullet B=\text{tr}(A^{\!\top} \!B)$.

\section{Low-rank semidefinite optimization}\label{sec:lr}
In the first part of the paper, we consider a (primal) semidefinite optimization problem
\begin{align}
    \label{eq:SDP}
        &\min_{X\in\S^m,\,x_{\text{lin}}\in\RR^m~} C \bullet X +d^\top x_{\text{lin}}\\
        &\text{{subject to}}\ \begin{aligned}[t]&A_i \bullet X+ (D^\top x_{\text{lin}})_i=b_i, \quad i=1,\dots,n\\
        &X\succeq 0,\  x_{\text{lin}}\geq 0
    \end{aligned}\nonumber
\end{align}
together with its Lagrangian dual
\begin{align}
    \label{eq:SDD}
        &\max_{y\in\RR^n,\;S\in\S^m,s_{\text{lin}}\in\RR^m~} b^\top y\\
        &\text{{subject to}}~\begin{aligned}[t]&\sum_{i=1}^{n}y_iA_i+S=C\\
        &Dy+s_{\text{lin}}=d\\
        &S\succeq 0\\
        &s_{\text{lin}}\geq 0\,.
    \end{aligned}\nonumber
\end{align}
Here $A_i\in\S^m$, $i=1,\ldots,n$, $C\in\S^m$, $b\in\RR^n$, $D\in\RR^{\nu\times n}$ and $d\in\RR^\nu$ are data of the problem. 

While the linear constraints can be formally included in the linear matrix inequality, we write them explicitly, as they will be treated differently in the algorithms below.
Later, in \Cref{sec:precon}, we will also use a more specific notation for problems with several matrix variables. 

\subsubsection*{Basic assumptions}
We make the following assumptions about \eqref{eq:SDP}.

\begin{Assumption}\label{assum:a2}
	Problems \cref{eq:SDP,eq:SDD} are strictly feasible, i.e., there exist $X\in\S^m_{++}$, $y\in\RR^n$, $S\in\S^m_{++}$, such that $A_i \bullet X=b_i $ and $\sum_{i=1}^{n}y_iA_i+S=C$  (Slater's condition).
\end{Assumption}
\begin{Assumption}\label{assum:a3}
	Define the matrix $\BA=[\text{\rm vec}A_1,\dots, \text{\rm vec}A_n]$. We assume that matrix-vector products with $\BA$ and $\BAT $ may each be computed in $\mathcal{O}(n)$ flops and memory.
\end{Assumption}
\begin{Assumption}\label{assum:a4} The inverse $(D^\top D)^{-1}$ and matrix-vector product with $(D^\top D)^{-1}$ may each be computed in $\mathcal{O}(n)$ flops and memory.
\end{Assumption}
\begin{Assumption}\label{assum:a5} The dimension of $X$ is much smaller than the number of constraints, i.e., $m\ll n$.
\end{Assumption}

\begin{remark}
    \Cref{assum:a4} is satisfied, for instance, in case of box constraints on variable $y$ in the dual formulation.
    \Cref{assum:a5} is not really needed as a ``mathematical'' assumption, all methods presented in the paper would work without it. However, the complexity of these methods would then be comparable with standard software using Cholesky factorization. The presented methods will be only superior to the standard software under this assumption.
\end{remark}

\subsubsection*{Low-rank assumption}
As mentioned in the Introduction, our approach is focused on linear systems with matrices having few large outlying eigenvalues. This, of course, includes matrices with very low rank but also matrices with ``approximate'' low rank. These matrices frequently appear in optimization algorithms, such as interior-point methods, when the ``exact'' low rank is only attained at the (unreached) optimum. Moreover, the approach allows us to solve problems with matrices of high rank but with a few large outlying eigenvalues, such as the problems in \Cref{sec:truss}.

Because our approach covers SDP problems with (genuinely) low-rank solutions and for the lack of better terminology for the larger class of matrices with outlying eigenvalues, we still call our main assumption ``low-rank'' and the problems we are targeting ``SDP problems with low-rank solutions''. 

Our main assumption used in the design of an efficient preconditioner concerns the eigenvalue distribution of the solution $X^*$ of \eqref{eq:SDP}. We will assume that $X^*$  has a very small number of large outlying eigenvalues. This will include problems where $X^*$ is of a very low rank.
\begin{Assumption}\label{assum:alr}
Let $X^*$ be the solution of \eqref{eq:SDP}. We assume that $X^*$ has $k$ outlying eigenvalues, i.e., that
$$
    (0\leq)\;\lambda_1(X^*)\leq\cdots\leq\lambda(X^*)_{m-k}\ll \lambda(X^*)_{m-k+1}\leq\cdots\leq\lambda(X^*)_m\,,
$$
where $k$ is very small, typically smaller than 10 and, often, equal to 1.
\end{Assumption}

\begin{remark}
In order to emphasize the difference of this paper to \cite{Zhang_2017}---that otherwise served as our motivation---we include the following remarks.
\begin{enumerate}
    \item We do not assume that the matrix $\BA^T \BA$ is easily invertible. Indeed, it is \emph{not} in our application in \Cref{sec:truss}, contrary to examples in \cite{Zhang_2017} where this matrix is diagonal.
    \item We do not assume anything about the distribution of the first $(m-k)$ eigenvalues of~$X^*$; they may not be equal to zero, they may not be clustered.
    \item We make no assumption about the ``rank'' of the solution $x^*_{\rm lin}$ of \eqref{eq:SDP}, i.e., about the number of active constraints in the linear inequality $Dy\leq d$ in \eqref{eq:SDD}. This can be arbitrarily large or small.
\end{enumerate}
\end{remark}

\section{Generalized augmented Lagrangian method}\label{Penalty/Barrier}
\subsection{Primal approach}
In this section, an augmented Lagrangian type algorithm for the solution of SDPs of type \cref{eq:SDD} is presented. Our method is based on the PENNON algorithm in \cite{PENNON}, but implements a number of useful modifications, which help to deal with ill-conditioned data. In our algorithm, we consider the LMI in \cref{eq:SDD} without slack, i.e., 
\begin{align}
A^{\rm{lmi}}(y)&=A_0(y)-C \preceq 0,\label{eq:LMI}
\end{align}
where $A_0(y)=\sum^n_{i=1}A_iy_i$. Moreover, we formally treat the linear constraints as a diagonal matrix inequality, i.e.,
\begin{align*}
A^{\rm{lin}}(y)&=\textrm{diag}(Dy-d) \preceq 0,
\end{align*}
and set $\tilde{b} := -b$ in order to turn \cref{eq:SDD} into a minimization problem. In a general setting, we may have many constraints of type \cref{eq:LMI}, however, for the sake of readability in this section, we assume that $A(y)$ is a block matrix with only two blocks 
\begin{align*}
A(y)=
\begin{pmatrix}
 A^{\rm{lmi}}(y) &0  \\
 0 & A^{\rm{lin}}(y)
 \end{pmatrix}.
\end{align*}
Here, $A^{\rm{lmi}}(y) \in \mathbb{S}^{m}$ and $A^{\rm{lin}}(y) \in \mathbb{S}^{\nu}$. Moreover $A^{\rm{lin}}(y)$ is a diagonal matrix for every $y \in  \mathbb{R}^n$ by construction, while we do not assume any particular structure for $A^{\rm{lmi}}(y)$. 

The Lagrangian for the resulting problem is defined by the formula
\begin{equation}
  L(y,X)=\tilde{b}^\top y+A(y) \bullet X = \tilde{b}^\top y+A^{\rm{lmi}}(y) \bullet X^{\rm{lmi}} + A^{\rm{lin}}(y) \bullet X^{\rm{lin}},
\end{equation}
where $X$ is a Lagrange multiplier in $\mathbb{S}^{m+\nu}_{+}$, which is again a block matrix composed from the blocks $X^{\rm{lmi}} \in \mathbb{S}^{m}_{+}$ and $X^{\rm{lin}} \in \mathbb{S}^{\nu}_{+}$. 
Next, a class of scalar-valued penalty functions is defined.
\begin{definition} \label{def:pen}
Let $\varphi:\mathbb{R}\to \mathbb{R}$ be a real-valued penalty function having the following properties \cite{Zibulevsky}:
\begin{enumerate}
    \item $\varphi$ is a strictly decreasing twice differentiable strictly convex function of one variable with Dom $\varphi=(b,\infty), -\infty\leq b <0$,
    \item
    $\lim_{t \rightarrow b} \varphi^{\prime} (t)=-\infty$,
    \item
    $\lim_{t \to \infty} \varphi^{\prime} (t)=0$,
    \item
    $\varphi(0)=0$,
    \item $\varphi^{\prime} (0)=-1$.
\end{enumerate}
\end{definition}
Using \cref{def:pen} and a positive penalty parameter $\pi$, we define a whole family of penalty functions
\begin{equation}
\varphi_\pi(t)=\pi\,\varphi\left(\frac{t}{\pi}\right),
\label{eq:fampen}
\end{equation}
for which we observe:
\[ \lim_{\pi\to 0} \varphi_\pi(t) =
  \begin{cases}
    \ 0,  & \quad t \leq 0, \\
    \ \infty,  & \quad t > 0.
  \end{cases}
\]
In this paper, to penalize positive diagonal values of $A^{\rm{lin}}(y)$, we use the so called \emph{quadratic-logarithmic} penalty function \cite{Zibulevsky,YouzefoZibulevsky,BenTalZibulevsky,PHDThesisZibulevsky}  which is given by 

\[ \varphi^{\rm{qlog}}(t) =
  \begin{cases}
       -\log(1-t),   & \quad  t\leq \tau,\\
   \frac{(1+\tau)^3(-t-\tau)^4}{4} +\frac{1}{3}\left(\frac{t+\tau}{1+\tau}\right)^3+\frac{1}{2}\left(\frac{-t-\tau}{1+\tau}\right)^2-\log(1+\tau),   & \quad  t > \tau,
  \end{cases}
\]
where $\tau$ is a positive number in which the logarithmic function is extrapolated in a twice continuously differentiable way by a quadratic function. 

For the linear matrix inequality constraints, i.e., the block $A^{\rm{lmi}}(y)$ we consider the \emph{hyperbolic} penalty function defined as:
\begin{equation}
    \varphi^{\rm{hyp}}(t)=-\frac{1}{t-1}-1.
    \label{eq:hyperbolic}
\end{equation}
It is easy to check that properties (1--5) of \cref{def:pen} are satisfied for both functions $\varphi^{\rm{qlog}}, \varphi^{\rm{hyp}}$ \cite{Zibulevsky,Polyak,DiplomaMS}.\\
Next we explain, how a penalty function from \cref{def:pen} can be used to penalize positive eigenvalues of $A^{\rm{lmi}}(y)$.
\begin{definition}
Let $\varphi:\mathbb{R}\to \mathbb{R}$ be a given function and $A\in \mathbb{S}^m$ be a given symmetric matrix. Let $A=S\Lambda S^\top$, with $\Lambda =diag(\lambda_1,\ldots,\lambda_m)^\top$ being an eigenvalue decomposition of $A$. Then the \emph{primary matrix function} $\Phi_\pi$ corresponding to $\varphi_\pi$ is defined by \cite{horn1991topics}:
\label{DefPrimarymatrix}
\end{definition}
\begin{align*}
\Phi_\pi : \mathbb {S}^m &\to \mathbb {S}^m\\
A&\to S
\begin{pmatrix}
 \varphi_\pi(\lambda_1) &0 & \cdots &0 \\
 0 &\varphi_\pi(\lambda_2)  & \, & \vdots \\
  \vdots  & \,  & \ddots & 0  \\
  0 & \cdots & 0 & \varphi_\pi(\lambda_m)
 \end{pmatrix}
S^\top.
\end{align*}


Using this, in principle, any scalar penalty function can be turned into a matrix penalty function. The main reason why we are using the hyperbolic penalty function \cref{eq:hyperbolic} is that for the evaluation of it, no eigenvalue decomposition is required. Instead, we can write \cite{PENNON}:
\begin{align*}
    \Phi^{\rm{hyp}}_\pi\left(A^{\rm{lmi}}(y)\right)&=\pi\left(-\left(\pi^{-1}A^{\rm{lmi}}(y)-I\right)^{-1}-I\right)=\pi^2\mathcal{Z}(y,\pi)-\pi I,
\end{align*}
where $\mathcal{Z}(y,\pi)=-(A^{\rm{lmi}}(y)-\pi I)^{-1}$ and $I$ is the identity matrix in $\mathbb{S}^{m}$. Note that we formally also use the formalism of primary matrix functions for the diagonal block $A^{\rm{lin}}(y)$ giving rise to a matrix penalty function $\Phi^{\rm{qlog}}_\pi$, but of course in this case the scalar penalty function can be directly applied to the diagonal entries of $A^{\rm{lin}}(y)$, which means that also here the eigenvalue decomposition can be avoided. In the following, we use the compact notation $\Phi_\pi(A(y))$ with $\pi = (\pi^{\rm{lin}},\pi^{\rm{lmi}})$ for the expression
\begin{align*}
\begin{pmatrix}
 \Phi^{\rm{hyp}}_{\pi^{\rm{lmi}}} (A^{\rm{lmi}}(y)) &0  \\
 0 & \Phi^{\rm{qlog}}_{\pi^{\rm{lin}}} (A^{\rm{lin}}(y))
 \end{pmatrix}.
\end{align*}

Using the penalty functions $\Phi^{\rm{qlog}}_\pi$ and $\Phi^{\rm{hyp}}_\pi$, we are now able to rewrite \cref{eq:LMI} as:
\begin{equation}
    A(y) \preceq 0\Leftrightarrow \Phi_\pi\left(A(y)\right) \preceq 0\Leftrightarrow \Phi^{\rm{hyp}}_{\pi^{\rm{lmi}}}\left(A^{\rm{lmi}}(y)\right) \preceq 0 \ \rm{and} \ \Phi^{\rm{qlog}}_{\pi^{\rm{lin}}} \left(A^{\rm{lin}}(y)\right) \preceq 0 .
    \label{eq:nonrescale}
\end{equation}
We note that this reformulation is sometimes referred to as \emph{nonlinear rescaling} \cite{GrivaPolyak,NonlinPolyak}. Using this and the Lagrange formalism, we can now define the \emph{generalized augmented Lagrangian} associated with \cref{eq:SDP} as:
\begin{align}
\label{eq:ModifiedLagrange1}
    & F(y,\bar{y},X,\pi^{\rm{lin}},\pi^{\rm{lmi}}) =\tilde{b}^\top y+\frac{r}{2}\|y-\bar{y}\|^2 + X \bullet \Phi_\pi\left(A(y)\right) \\
    &=\tilde{b}^\top y+\frac{r}{2}\|y-\bar{y}\|^2+X^{\rm{lmi}} \bullet \Phi^{\rm{hyp}}_{\pi^{\rm{lin}}} \left(A^{\rm{lmi}}(y)\right) + X^{\rm{lin}} \bullet \Phi^{\rm{qlog}}_{\pi^{\rm{lmi}}}\left(A^{\rm{lin}}(y)\right).\nonumber
\end{align}
In contrast to the PENNON algorithm \cite{PENNON}, in \cref{eq:ModifiedLagrange1}, we add a proximal point regularization term with a regularization parameter $r$, which provides a uniform lower bound for the eigenvalues of the Hessian $\nabla^2_{yy}F$, as will be seen later.

Next, the basic (primal) augmented Lagrangian algorithm is defined in \cref{PBMAlgorithm}.
\begin{algorithm}
\caption{Generalized Augmented Lagrangian Algorithm}
\label{PBMAlgorithm}
 Let $X_0\in \mathbb{S}^{m+\nu}_{++}$, $\pi^{\rm{lin}}_0>0$, $\pi^{\rm{lmi}}_0>0$ and $y_0\in \mathbb{R}^n$ be given.\\[-1em]
 \begin{algorithmic}[1]
 \For {$k=0,1,2,\ldots$}
   \State\label{a1l2}  Minimize \cref{eq:ModifiedLagrange1} with respect to $y$ with fixed $X_k$ and $\pi_k$, i.~e. compute $y_{k+1} \in \mathbb{R}^n$ such that: \begin{equation} 
       \nabla_y F(y_{k+1},y_k,X_k,\pi_k^{\rm{lin}},\pi_k^{\rm{lmi}})=0.
       \label{eq:argmin}
   \end{equation} 
    \State\label{a1l3} Update the multiplier $X_k$: 
    \medskip
    \begin{itemize}
        \item for the linear constraint block compute the diagonal multiplier $X^{\rm{lin}}_{k+1}$ by  
        \begin{equation}
        \widetilde{X}^{\rm{lin}}_{k+1}=\bar{X}^{\rm{lin}}(y_{k+1}, X^{\rm{lin}}_k, \pi_k^{\rm{lin}}),
        \end{equation}
        with the \emph{multiplier update function} $$\bar{X}_{i,j}^{\rm{lin}}(y,X,\pi)=
          \begin{cases}
    \ X_{i,i} \varphi^{\prime}_{\pi} (A_{i,i}^{\rm{lin}}(y)),  & \quad i = j, \\
    \ 0,  & \quad i \neq j;
  \end{cases}$$
         \item for the LMI block compute
        \begin{equation}
         \widetilde{X}^{\rm{lmi}}_{k+1}=\bar{X}^{\rm{lmi}}(y_{k+1}, X^{\rm{lmi}}_k, \pi_k^{\rm{lmi}}),
        \end{equation}
    with the \emph{multiplier update function} $\bar{X}^{\rm{lmi}}(y,X,\pi)=\pi^2\mathcal{Z}(y,\pi) X \mathcal{Z}(y,\pi)$, see \cite{PENNON};
        \label{eq:updatemulitiplier}
        \item damp multiplier update by 
         \begin{equation*}
         X^{\rm{lin}}_{k+1} = (1-\gamma^{\rm{lin}}) \widetilde{X}^{\rm{lin}}_{k+1} + \gamma^{\rm{lin}} \widetilde{X}^{\rm{lin}}_{k+1}, \quad X^{\rm{lmi}}_{k+1} = (1-\gamma^{\rm{lmi}}) \widetilde{X}^{\rm{lmi}}_{k+1} + \gamma^{\rm{lmi}} \widetilde{X}^{\rm{lmi}}_{k+1}
        \end{equation*}       
        using algorithmic parameters $0 \leq \gamma^{\rm{lin}}, \gamma^{\rm{lmi}} \leq 1$.
   \end{itemize}
\medskip
     \State\label{a1l4} Update the penalty parameters $\pi^{\rm{lin}}_{k}, \pi^{\rm{lmi}}_{k}$:
\medskip
    \begin{itemize}
        \item for the linear constraint block compute 
    \begin{equation*}
    \pi^{\rm{lin}}_{k+1}=\max(\pi^{\rm{lin}}_{\rm{min}}, \pi_{\rm{upd}}^{\rm{lin}} \pi^{\rm{lin}}_k);
    \end{equation*}
    \end{itemize}
    \begin{itemize}
        \item for the linear matrix inequality constraints
        \begin{equation*}
        \pi^{\rm{lmi}}_{k+1}=\max\left(\pi^{\rm{lmi}}_{\rm{min}}, \pi_{\rm{upd}}^{\rm{lmi}} \pi^{\rm{lmi}}_k, 1.01 \lambda_{\max}(A^{\rm{lmi}}(y_{k+1})) \right),
        \end{equation*}
        where $\lambda_{\max}(\cdot)$ denotes the largest eigenvalue of a matrix
    and $0<\pi_{\rm{upd}}^{\rm{lmi}},\pi_{\rm{upd}}^{\rm{lin}}<1$, $\pi^{\rm{lmi}}_{\rm{min}}, \ \pi^{\rm{lin}}_{\rm{min}} > 0$ are algorithmic parameters.
    \end{itemize}
    \EndFor
 \end{algorithmic}
\end{algorithm}
As usual in augmented Lagrangian methods, the Lagrange multiplier update in \cref{eq:updatemulitiplier} is motivated as follows.
From \cref{eq:argmin}, we have
\begin{align} \label{eq:gradaugy}
&    \nabla_y F(y_{k+1},y_k,X_k,\pi_k^{\rm{lin}},\pi_k^{\rm{lmi}})=\\ &\qquad\tilde{b}+r(y_{k+1}-y_k)+A^*\bar{X}(y_{k+1},X_k,\pi_k^{\rm{lin}},\pi_k^{\rm{lmi}})=0\,; \nonumber
\end{align}
here, $A^*$ is the adjoint operator to $A$ and $\bar{X}(y_{k+1},X_k,\pi_k^{\rm{lin}},\pi_k^{\rm{lmi}})$ is a block diagonal matrix with blocks $\bar{X}^{\rm{lin}}(y_{k+1},X_k^{\rm{lin}},\pi_k^{\rm{lin}})$ and $\bar{X}^{\rm{lmi}}(y_{k+1},X_k^{\rm{lmi}},\pi_k^{\rm{lmi}})$; see \cref{PBMAlgorithm}.
Now, inserting the updated multiplier $X_{k+1}$, the left hand side of the last equation becomes $\tilde{b}+r(y_{k+1}-y_k)+A^*X_{k+1}$. For $r=0$, this implies 
\begin{equation}
    \nabla_y L(y_{k+1},X_{k+1})=\tilde{b}+A^*X_{k+1}=0,
    \label{eq:gradLy}
\end{equation}
which means that $y_{k+1}$ is also stationary for the standard Lagrangian, if we set the proximal point parameter equal to zero. For positive $r$, this holds at least in an asymptotic sense, as $y_{k+1}-y_k \to 0$ for $k\to \infty$.



When we solve Step~2 of \cref{PBMAlgorithm} by Newton's method, \cref{PBMAlgorithm} is (apart from the proximal point regularization and the special treatment of the linear constraints) essentially equivalent to the PENNON algorithm described in \cite{PENNON}.
%
%
Before applying a few further modifications to \cref{PBMAlgorithm}, let us have a more detailed look at the Newton iteration. For this, we need to compute the gradient and Hessian of the augmented Lagrangian \cref{eq:ModifiedLagrange1} in each iteration. One easily verifies that the Hessian can be computed using the formula
%
\begin{align} \label{eq:Hessianaug}
    &H_k(y) :=\nabla^2_{yy}F(y,y_k,X_k,\pi_k^{\rm{lin}},\pi_k^{\rm{lmi}}) \\
    &=rI+ 2 \left[ A_i\bullet\bar{X}^{\rm{lmi}}(y,X_k^{\rm{lmi}},\pi_k^{\rm{lmi}}) A_j \mathcal{Z}(y,\pi_k^{\rm{lmi}})\right)]_{i,j=1}^n
   + D^\top \bar{W}^{\rm{lin}}(y,X_k^{\rm{lin}},\pi_k^{\rm{lin}})D \nonumber \\
    &= rI+ 2 \BAT \left(\bar{X}^{\rm{lmi}}(y,X_k^{\rm{lmi}},\pi_k^{\rm{lmi}}) \otimes \mathcal{Z}(y,\pi_k^{\rm{lmi}})\right)\BA + D^\top \bar{W}^{\rm{lin}}(y,X_k^{\rm{lin}},\pi_k^{\rm{lin}})D \nonumber
\end{align}
       with 
       \begin{equation}\label{eq:Wlin}
       \bar{W}_{i,j}(y,X,\pi)=
          \begin{cases}
    \ X_{i,i} \varphi^{\prime\prime}_{\pi} (A_{i,i}^{\rm{lin}}(y)),  & \quad i = j, \\
    \ 0,  & \quad i \neq j,
  \end{cases}
  \end{equation}
  $\BA$ being the $m\times n$ matrix defined in \cref{assum:a3} and $D$ being the $\nu \times n$ matrix defining the linear constraints.

Let us first look at the conditioning 
of the Hessian.
\begin{lemma}
    The eigenvalues of $H_k(y)$ are bounded from below by $r$. 
\end{lemma}
\begin{proof}
Due to the positive semidefiniteness of the matrices $\bar{X}^{\rm{lmi}}(y,X_k^{\rm{lmi}},\pi_k^{\rm{lmi}})$, $\mathcal{Z}(y,\pi_k^{\rm{lmi}})$
and $\bar{W}^{\rm{lin}}(y,X_k^{\rm{lin}},\pi_k^{\rm{lin}})$, 
the second and third term in \cref{eq:Hessianaug} are positive semidefinite. Therefore, it can be directly seen that $r$ is a uniform lower bound for the smallest eigenvalue of $H_k(y)$. 
\end{proof}
\begin{remark}
If the penalty parameter sequences $(\pi_k^{\rm{lmi}})$ and $(\pi_k^{\rm{lin}})$ are bounded by $(\pi_*^{\rm{lmi}})$ and $(\pi_*^{\rm{lin}})$, respectively, and $(y_k,X_k)$ converges to a KKT pair $(y_*,X_*)$ of \cref{eq:SDP}, one can also show that all the eigenvalues of $H_k(y)$ are bounded from above for all $k$. This is seen from the following limits:
\begin{align*}
\bar{X}^{\rm{lmi}}(y_k,X_k,\pi_k^{\rm{lmi}}) &\to X_*^{\rm{lmi}} \ \mbox{ for }\ k \to \infty,\\
\mathcal{Z}(y_k,\pi_k^{\rm{lmi}})  &\to Z_* \ \mbox{ for }\ k \to \infty,\\ 
\bar{W}^{\rm{lin}}(y_k,X_k^{\rm{lin}},\pi_k^{\rm{lin}})&\to W_* \ \mbox{ for }\ k \to \infty,
\end{align*}
where $Z_*$ is a matrix with eigenvalues between 0 and $1/\pi_*^{\rm{lmi}}$ and $W_*$ is a matrix with eigenvalues between 0 and $(2/\pi_*^{\rm{lin}}) \lambda_{\max}(X_*^{\rm{lin}})$.
We note that we do not present a formal proof for this here, as this would require a full convergence proof for \cref{PBMAlgorithm}. For the latter, we refer to \cite{MSThesis}.
\end{remark}


Next, we would like to have a look at the complexity of a Newton step. While the cost of the solution of the Newton system is generally of order $\mathcal{O}(n^3)$, the costs for the assembly of the gradient and the Hessian of the augmented Lagrangian are dominated by the assembly of the LMI terms. 
For these, we note that each evaluation of the term $\bar{X}_k^{\rm{lmi}}(y,X,\pi) A_j \mathcal{Z}(y,\pi)$ is of order $\mathcal{O}(m^3)$. This step has to be carried out for every $A_j, \ j=1,\ldots,n$. Then for each such matrix, the inner product with all matrices $A_j, \ j=1,\ldots,n$ has to be computed. Thus, the total complexity of the assembly is $\mathcal{O}(nm^3)$. We note that for very sparse constraint matrices, a different way of assembling the Hessian can be chosen and the complexity formula reduces to $\mathcal{O}(n^2)$.
For more details, we refer to \cite{MKMS}.


Obviously, for very large SDP problems, assembling the Hessian in \cref{eq:Hessianaug} is very expensive and sometimes the matrix cannot even be stored. Therefore, in this paper, we suggest to use the PCG method to solve the Newton system. For this, instead of assembling the full Hessian, only matrix-vector multiplications are required. Given a vector $d \in \mathbb{R}^n$, these can be computed as
\begin{equation*}
H_k(y) d = rd + (A^{\rm{lmi}})^*Q(y,d) + D^\top \bar{W}^{\rm{lin}}(y,X_k^{\rm{lin}},\pi_k^{\rm{lin}})D d,
\end{equation*}
where 
\begin{equation*}
Q(y,d)=\bar{X}^{\rm{lmi}}(y,X^{\rm{lmi}}_k,\pi^{\rm{lmi}}_k)A_0(d)\mathcal{Z}(y,\pi^{\rm{lmi}}_k)\,.
\end{equation*}
Using this approach and assuming sparse data matrices, the complexity of Hessian-vector multiplication reduces to $\mathcal{O}(n+m^3)$. Of course, the total complexity of the (approximate) solution of the Newton system also depends on the number of PCG iterations as well as the cost of the assembly of the preconditioner. For further discussion on this topic, we refer to \Cref{sec:examples}.

%
%
%
%
%
\subsection{Primal-dual approach}
In numerical experiments, partly also performed by the original PENNON algorithm \cite{PENNON}, it was found that it is sometimes very hard to solve the subproblems in Step~2 of the \cref{PBMAlgorithm} to the required precision. This is particularly problematic if an iterative solver is used for the Newton system.

This was our motivation to apply one more modification, which is inspired by the work of Griva and Polyak \cite{GrivaPolyak}. We call this approach a \emph{Primal Dual Augmented Lagrangian} (PDAL) approach. To unburden the notation, we will not distinguish between the LMI and the linear constraint block as in the previous section. Moreover, as we will focus on Step~2 of \cref{PBMAlgorithm}, we skip the outer iteration index $k$ from all quantities in the sequel. To distinguish between variables which are fixed and variables which are optimized in Step~2, we add a tilde to all symbols which belong to fixed variables.

Now, the basic idea is to rephrase equation \cref{eq:argmin} 
\begin{equation}
    \nabla_y F(\widehat{y},\Tilde{y},\Tilde{X},\Tilde{\pi})=0,
    \label{eq:optimalitycondition}
\end{equation}
which is rewritten here as explained above, as the following primal-dual system:
\begin{align}
    \tilde{b} + r(\widehat{y}-\Tilde{y})+A^*\widehat{X}&=0
    \label{eq:Primalupdate}\\
    \widehat{X}-\bar{X}(\widehat{y})&=0.
    \label{eq:Dualupdate}
    \end{align}
      Here $\widehat{y}$ and $\widehat{X}$ are the variables and $\bar{X}(\widehat{y})=\Tilde{\pi}^2\mathcal{Z}(\widehat{y},\Tilde{\pi})\Tilde{X}\mathcal{Z}(\widehat{y},\Tilde{\pi})$ is the multiplier update function, which has been already introduced in the previous section. For the sake of readability, we will skip the second argument of $\mathcal{Z}$ in the remainder of this section.
      It is readily seen that system \cref{eq:Primalupdate}--\cref{eq:Dualupdate} is equivalent to \cref{eq:optimalitycondition}. In particular, if \cref{eq:Primalupdate}--\cref{eq:Dualupdate} is solved exactly, the solution $\widehat{y}_*$ solves \cref{eq:optimalitycondition} and $\widehat{X}_*$ corresponds to the (already updated) multiplier. Thus, at the first glance, the introduction of the additional variable $\widehat{X}$ just seems to make Step~2 of \cref{PBMAlgorithm} more expensive. However, it turns out that there are a number of advantages of the rephrased system. First, \cref{eq:Primalupdate} is---due to its linear structure---much easier to solve than \cref{eq:optimalitycondition} and, second, in every iteration we get a candidate to update both, the primal and the dual variable. The latter is particularly interesting at the later (outer) iterations, because, if we are sufficiently close to a solution of \cref{eq:SDP}, it can be expected that ultimately only one iteration will be required before $\Tilde{X}$ can be updated (see  \cite{GrivaPolyak} for a theoretical justification of this claim in the nonlinear programming case). This requires of course a relaxation of the stopping criterion applied in the inner iteration. We will return to this point a little later. 
      
Formally, we again apply Newton's method to solve system \cref{eq:Primalupdate}--\cref{eq:Dualupdate}. For this, we define
    \begin{align}
    G_1(\widehat{y},\widehat{X})&=\tilde{b}+r(\widehat{y}-\Tilde{y})+A^*\widehat{X}, \label{eq:G1}\\
    G_2(\widehat{y},\widehat{X})&=\widehat{X}-\bar{X}(\widehat{y}).\label{eq:G2}
\end{align}
Now, in order to derive the corresponding Newton system, we need to compute the Jacobians of $G_1$ and $G_2$. These are provided by the following lemma.
\begin{lemma}
Let $G_1$ and $G_2$ be given as in \cref{eq:G1} and \cref{eq:G2} and let $(\Delta \widehat{y},\Delta \widehat{X})$ be a direction in $\mathbb{R}^n \times \mathbb{S}^m$. Then the following formulas hold for the directional derivatives of $G_1$ and $G_2$:
\begin{align}
   	\partial_{\widehat{y}} G_1(\widehat{y},\widehat{X})[\Delta \widehat{y}]&=rI\Delta \widehat{y}=r\Delta \widehat{y},\\
   	\partial_{\widehat{y}} G_2(\widehat{y},\widehat{X})[\Delta \widehat{y}]&=
   	-\bar{X}(\widehat{y})A_0(\Delta\widehat{y})\mathcal{Z}(\widehat{y})-\mathcal{Z}(\widehat{y})A_0(\Delta\widehat{y})\bar{X}(\widehat{y}),\\
    	\partial_{\widehat{X}} G_1(\widehat{y},\widehat{X})[\Delta \widehat{X}]&=A^*\Delta\widehat{X},\\
     	\partial_{\widehat{X}} G_2(\widehat{y},\widehat{X})[\Delta \widehat{X}]&=\Delta\widehat{X}. 
\end{align}
\label{eq:jacobianG}
\end{lemma}
\begin{proof}
All formulas can be obtained by elementary operations.
\end{proof}

Now, introducing an inner iteration index $\ell$ and using \cref{eq:jacobianG}, we obtain the following primal-dual Newton system:
\begin{gather}
  r\Delta \widehat{y}^\ell+A^*\Delta\widehat{X}^\ell=-\tilde{b}-r(\widehat{y}^\ell-\bar{y})-A^*\widehat{X}^\ell, \label{eq:deltaydirection}\\
  -\bar{X}(\widehat{y}^\ell)A_0(\Delta\widehat{y}^\ell)\mathcal{Z}(\widehat{y}^\ell)-\mathcal{Z}(\widehat{y}^\ell)A_0(\Delta\widehat{y}^\ell)\bar{X}(\widehat{y}^\ell)+\Delta \widehat{X}^\ell=-\widehat{X}^\ell+\bar{X}(\widehat{y}^\ell).
  \label{eq:deltaxdirection}
\end{gather}
In order to solve \cref{eq:deltaydirection}--\cref{eq:deltaxdirection} efficiently, we suggest to substitute $\Delta \widehat{X}^\ell$ from \cref{eq:deltaxdirection} into \cref{eq:deltaydirection}. Then, on the left-hand side of \cref{eq:deltaydirection}, we obtain:
\begin{equation}
     r\Delta \widehat{y}^\ell+2A^*\left(\bar{X}(\widehat{y}^\ell)A_0(\Delta\widehat{y}^\ell)\mathcal{Z}(\widehat{y}^\ell)\right).
     \end{equation}
     Obviously, this is just $\nabla_{\widehat{y}\widehat{y}} F(\widehat{y}^\ell,\Tilde{y},\Tilde{X},\Tilde{\pi})\Delta\widehat{y}^\ell$, i.e., the Hessian of the augmented Lagrangian determined in the previous section, multiplied by $\Delta\widehat{y}^\ell$. We abbreviate this Hessian by $H^\ell$ in the sequel.
Similarly, on the right-hand side of \cref{eq:deltaydirection}, the substitution yields:
   \begin{align}
     -\tilde{b}-r(\widehat{y}^\ell-\bar{y})-A^*\widehat{X}^\ell+A^*\left(\widehat{X}^\ell-\bar{X}(\widehat{y}^\ell)\right)
     &=-\tilde{b}-r(\widehat{y}^\ell-\Tilde{y})-A^*\bar{X}(\widehat{y}^\ell)\nonumber \\
     &=-\nabla_{\widehat{y}} F(\widehat{y}^\ell,\Tilde{y},\Tilde{X},\tilde{\pi}).
\end{align}
Thus, we can rewrite \cref{eq:deltaydirection} and \cref{eq:deltaxdirection} as:
\begin{align}
    H^\ell \Delta\widehat{y}^\ell&=-\nabla_{\widehat{y}} F(\widehat{y}^\ell,\Tilde{y},\Tilde{X},\tilde{\pi}),
    \label{eq:N1}\\
    \Delta \widehat{X}^\ell&=-\widehat{X}^\ell+\bar{X}(\widehat{y}^\ell) + \bar{X}(\widehat{y}^\ell)A_0(\Delta\widehat{y}^\ell)\mathcal{Z}(\widehat{y}^\ell)+\mathcal{Z}(\widehat{y}^\ell)A_0(\Delta\widehat{y}^\ell)\bar{X}(\widehat{y}^\ell).\label{eq:N2}
\end{align}
We can first solve \cref{eq:N1}, which is precisely the Newton system introduced in \Cref{Penalty/Barrier} and then just evaluate \cref{eq:N2}. 
In order to damp the Newton step, we introduce the  \emph{merit function}:
\begin{equation}
    M(\widehat{y},\widehat{X})=\frac{1}{2}\left(\|G_1(\widehat{y},\widehat{X})\|^2+\|G_2(\widehat{y},\widehat{X})\|^2\right)\,.
    \label{eq:merit}
\end{equation}
With this, we present \cref{PDAlgorithm} for the solution of the primal-dual system.
\begin{algorithm}
\caption{Newton's method for the primal-dual system}
\label{PDAlgorithm}
Let $\widehat{y}^0=\Tilde{y}, \widehat{X}^0=\Tilde{X}$ and a convergence tolerance $\varepsilon>0$ be given. Let further $(\Tilde{y},\Tilde{X})$ be known from the current outer iteration. 
\begin{algorithmic}[1]
\For{$\ell=0,1,2,\ldots$}
\State\label{a2l2} Solve \cref{eq:N1}--\cref{eq:N2} to find the primal-dual Newton direction $(\Delta\widehat{y}^\ell,\Delta\widehat{X}^\ell)$.
\State\label{a2l3}  Backtracking line search: compute a step length $0 < \alpha \leq 1$ such that 
\begin{equation}
    M(\widehat{y}^\ell+\alpha\Delta\widehat{y}^\ell,\widehat{X}^\ell+\alpha\Delta\widehat{X}^\ell)\leq M(\widehat{y}^\ell,\widehat{X}^\ell) + 0.05\alpha M^\prime(\widehat{y}^\ell,\widehat{X}^\ell)[\Delta\widehat{y}^\ell, \Delta\widehat{X}^\ell],
    \label{eq:linesearch}
\end{equation}
where $M^\prime(\widehat{y}^\ell,\widehat{X}^\ell)[\Delta\widehat{y}^\ell, \Delta\widehat{X}^\ell]$ is the directional derivative of the merit function \cref{eq:merit} at $(\widehat{y}^\ell,\widehat{X}^\ell)$ in direction of $(\Delta \widehat{y}^\ell,\Delta\widehat{X}^\ell)$.
\State\label{a2l4}  Update the primal and dual iterates using the following formulas:
\begin{equation}
    \widehat{y}^{\ell+1}=\widehat{y}^\ell+\alpha\Delta\widehat{y}^\ell,\quad \widehat{X}^{\ell+1}=\widehat{X}^\ell+\alpha\Delta\widehat{X}^\ell.
\end{equation}
\State\label{a2l5}  If $M(\widehat{y}^{\ell+1},\widehat{X}^{\ell+1})\leq \varepsilon$ and $\widehat{X}^\ell$ s.p.d., STOP and return the approximate solution $y^s:=\widehat{y}^{\ell+1}$, $X^s:=\widehat{X}^{\ell+1}$
\EndFor
\end{algorithmic}
\end{algorithm}
\begin{remark}
     Note that in the first iterations it might happen that $\widehat{X}^\ell$ has negative eigenvalues. However, knowing that $\bar{X}(\widehat{y}^\ell)$ is always a positive definite matrix by construction, as soon as \cref{eq:Dualupdate} is solved well enough, the matrix $\widehat{X}^\ell$ will be positive definite as well. Thus, the second stopping criterion in Step~5 can always be satisfied.
     \end{remark}
     \begin{remark}
     As the linear system in Step~2 is formally fully equivalent to the Newton system solved in Step~2 of \cref{PBMAlgorithm}, we can apply the same direct or iterative solvers. In this article, the Newton system is solved approximately by the PCG method; see \Cref{sec:precon}.
      \end{remark}

  
Finally, we explain how precisely \cref{PDAlgorithm} is integrated into \cref{PBMAlgorithm}. We perform the following steps: First, the new algorithm is replacing the primal iteration in the first step of \cref{PBMAlgorithm}. Second, we introduce the following stopping criterion for the inner (primal-dual) iteration:
    \begin{equation} \label{eq:PDAL_earlystopping}
    \begin{aligned}
        & e(\widehat{y}^\ell,\widehat{X}^\ell) < 0.5 ~ e(\Tilde{y},\Tilde{X}) \ \land \|G_2(\widehat{y}^\ell,\widehat{X}^\ell)\|^2<0.1 \\ & \land \ \|G_1(\widehat{y}^\ell,\widehat{X}^\ell)\|^2< 0.05 \max\{1,\|\nabla_y F(\widehat{y}^\ell,\Tilde{X},\Tilde{\pi})\|\} \land \ \widehat{X}^\ell \succ 0.
\end{aligned}
        \end{equation}
        with the \emph{primal-dual error function} 
        $$
        e(\widehat{y}^\ell,\widehat{X}^\ell):=\max\{p_{\rm{feas}}(\widehat{y}^\ell,\widehat{X}^\ell),d_{\rm{feas}}(\widehat{y}^\ell,\widehat{X}^\ell),d_{\rm{gap}}(\widehat{y}^\ell,\widehat{X}^\ell)\}.
        $$
    Thus, in every inner iteration we also monitor the progress with respect to the outer stopping criterion, which we define as $e({y}^k,{X}^k) \leq \varepsilon$ with $\varepsilon$ denoting the outer stopping tolerance. It is seen that the early stopping criterion is only applied if also the primal-dual system is solved reasonably well and the candidate for the new matrix multiplier is positive definite.

With this, the final PDAL algorithm is presented as \cref{PDALAlgorithm}.
\begin{algorithm}
\caption{Generalized Primal-Dual Augmented Lagrangian Algorithm}
\label{PDALAlgorithm}
 Let $X_0\in \mathbb{S}^{m+\nu}_{++}$, $\pi^{\rm{lin}}_0>0$, $\pi^{\rm{lmi}}_0>0$ and $y_0\in \mathbb{R}^n$ be given. Choose stopping tolerances $\varepsilon>0$, $\varepsilon_{\scriptscriptstyle \rm DIMACS}>0$ and  $(\varepsilon_k) \subset \mathbb{R}_+$ for the outer and inner loop, respectively. \\[-1em]
 \begin{algorithmic}[1]
 \For {$k=0,1,2,\ldots$}
   \State Use \cref{PDAlgorithm} with $(\Tilde{y},\Tilde{X})=(y_{k},X_{k})$ and stopping tolerance $\varepsilon_k$ to compute a new primal and dual iterate $(y_{k+1},X_{k+1})$. In step 5 of \cref{PDAlgorithm} apply the additional early stopping criterion \cref{eq:PDAL_earlystopping}.
   \State STOP, if $e({y}_{k+1},{X}_{k+1}) < \varepsilon$ or the DIMACS criteria are satisfied with the precision $\varepsilon_{\scriptscriptstyle \rm DIMACS}$ . Return $({y}_{k+1},{X}_{k+1})$ as approximate primal-dual solution.
    \State Damp the multiplier update by carrying out step 3 of \cref{PBMAlgorithm}.
     \State Update the penalty parameters by carrying out  step 4 of \cref{PBMAlgorithm}.
    \EndFor
 \end{algorithmic}
\end{algorithm}

\section{Interior-point method} \label{sec:IP}
In this section, we will describe a primal-dual predictor-corrector interior-point method for SDP which is using the Nesterov–Todd (NT) direction. We will closely follow the articles by Todd, Toh and Tütüncü \cite{Todd-Toh} and by Toh and Kojima \cite{Toh-Kojima} and only repeat formulas needed to present the structure of the matrix that is required to develop the preconditioners in \Cref{sec:precon}. The basic framework of our interior-point solver thus, more or less, mimics that of the software SDPT3. 

We use the notation of problems \eqref{eq:SDP} and \eqref{eq:SDD}; however, for the sake of simplicity, we ignore the linear constraints in the following development. They can be, of course, formally included into the linear matrix inequality.

\subsection{Basic framework}
Let $\BA$ be an $m^2\times n$ matrix defined by
$$
\BA:=[\text{vec} A_1,\dots,\text{vec} A_n].
$$
Our algorithm follows the standard steps of a primal-dual interior-point method. We write down optimality conditions for \cref{eq:SDP} with relaxed complementarity condition:
\begin{align}\label{eq:OC}
\BA^\top \text{vec} (X) &= b,\\
\BA y-S& = C,\nonumber\\
XS &= \sigma\mu I\,,\nonumber
\end{align}
where $\mu=\frac{X\bullet S}{m}$ and $\sigma$ is a centering parameter to be specified below. This system of nonlinear equations is solved repeatedly until convergence. 

\subsection{Newton direction}
The system \cref{eq:OC} is solved approximately by Newton's method, where in every iteration of the method we solve  the following system of linear equations in variables $(\Delta X,\Delta y,\Delta S)$, the Newton direction: 
\begin{subequations}
	\begin{align}
	\BA^\top \text{vec}(\Delta X)&=r_p,\label{eq:a}\\
	\BA\Delta y +\text{vec}(\Delta S)&=\text{vec}(R_d),\label{eq:b}\\
	H_W(X\Delta S +\Delta X S)  &=  \sigma\mu I-H_W(XS)\,.\label{eq:c}
	\end{align}
\end{subequations}
Here
$$
r_p:=b-\BA^\top X,\quad R_d:=C-S-\sum_{i=1}^{n}y_iA_i,
$$
and $H_W$ is a linear transformation guaranteeing symmetry of the resulting matrix, in particular
$$
   H_W(M) = \frac{1}{2}\left(WMW^{-1}+W^{-T}M^TW^T\right)
$$
with some invertible matrix $W$  (see \cite{zhang}). The matrix $W\in \mathbb{S}_{++}^m$ is known as the \emph{scaling matrix} and there are different ways to construct it. By assuming $({X},{y},{S})$ to be the current iterate, the most studied choices of $W$ are the following:
\begin{itemize}
	\item 
	primal scaling, in which $W={{X}}$;
	\item
	dual scaling, in which $W={{S}}^{-1}$; 
	\item
	Nesterov-Todd (NT) scaling, in which 
	$$
		W = S^{-\frac{1}{2}}\left(S^{\frac{1}{2}}XS^{\frac{1}{2}}\right)^{\frac{1}{2}}S^{-\frac{1}{2}}
		= X^{\frac{1}{2}}\left(X^{-\frac{1}{2}}SX^{-\frac{1}{2}}\right)^{\frac{1}{2}}X^{\frac{1}{2}}
	$$
	satisfying ${X}=W{S}W$ and ${S}=WX^{-1}{}W$.
\end{itemize}
In our algorithm, we use the NT scaling.

As in the case of linear programming, instead of solving the linear system of $2m^2 + n$ equations \cref{eq:a}–\cref{eq:c} directly, we can solve a Schur complement equation (SCE) involving only $\Delta y$. The general bottleneck of any interior-point method for SDP is assembling and solving this SCE
\begin{subequations}
	\begin{align}
		&H\Delta y =r.\label{eq:scea}&
	\end{align}
	Here $H$ is the Schur complement matrix with elements
	\begin{align}
		&H_{ij}=A_i\bullet WA_j W, \quad i,j=1,\dots,n,\label{eq:sced}
	\end{align}
	and
	\begin{align}\label{eq:rhs_pre}
	r = r_p + \BAT \text{vec}\left(W R_d W + WSW\right)\,.
	\end{align}
	For details of computing $W$, $r$, $\Delta X$ and $\Delta S$ efficiently for the NT scaling, see \cite{Toh-Kojima}.
\end{subequations}

By considering \cref{eq:scea,eq:sced} we have the following linear system
\begin{align}
	(H\Delta y)_i= A_i\bullet\left[W\left(\sum_{j=1}^{n}\Delta y_i A_j\right)W\right]=r_i, \quad i=1,\dots,n.\label{eq:linearsys}
\end{align}
Vectorizing the matrix variables allows \cref{eq:linearsys} to be written as
\begin{align}
	\left(\BAT (W\otimes W) \BA\right)\Delta y=r\,. \label{eq:schurcomp}
\end{align}
\subsection{Interior-point algorithm}
In the following \cref{algo:IP}, we use a second-order correction to the NT direction. Suppose that we have approximations $\delta X$ and $\delta S$ to the search directions. The correction $R_{NT}$ is given by
\begin{align}
	R_{NT}=-\left(G^{-1}\left(\delta X\delta S\right)G+G^\top \left(\delta S\delta X\right)G^{-T}\right)./\left(de^\top +ed^\top \right),\label{RNT}
\end{align}
where \textquotedblleft$./$" is entry-wise division, $G$ satisfies $W=GG^\top$, $d:=\text{diag}(G^\top S G)$ and $e:=(1,\dots,1)^\top $. 
With this correction, we define a new right-hand side of \cref{eq:sced} as
	\begin{align}\label{eq:rhs_cor}
	r = r_p + \BAT \text{vec}\left(W R_d W + W S W - \sigma\mu S^{-1}  - GR_{NT}G^\top\right)\,. 
\end{align}
  \begin{algorithm}
\caption{NT-IP Algorithm}
\label{algo:IP}
Given an initial iterate $(X^0,y^0,S^0),$ with $X^0$, $S^0$ positive definite. Choose $\tau=0.9$. Let $\mu=\frac{X\bullet S}{m}$.
\begin{algorithmic}[1]
\For{$\ell=0,1,2,\ldots$}
	\State{\emph{Predictor step}:
	For $(X^\ell,y^\ell,S^\ell)$ compute the Newton direction $(\delta X,\delta y,\delta S)$ from \cref{eq:scea,eq:a,eq:b,eq:c} with right-hand side \cref{eq:rhs_pre}.}
	\State
	Find
	\begin{align*}
	\alpha:=\min\Big(1,	{\frac{-\tau}{{\lambda_{\min}} \left((X^\ell)^{-1}\delta X\right)}}\Big),\quad 	\beta:=\min\Big(1,	{\frac{-\tau}{{\lambda_{\min}} \left((S^\ell)^{-1}\delta S\right)}}\Big)\label{eq:alphabeta}
	\end{align*}
	  ensuring that $X^\ell + \alpha \delta X$ and $S^\ell + \beta \delta S$ are positive definite.
	Set
	\begin{align*}
	\sigma:=\Big[\frac{(X^\ell+\alpha\delta X)\bullet(S^\ell+\beta\delta S)}{X^\ell\bullet S^\ell}\Big]^3.
	\end{align*}
	  \State
	  \emph{Corrector step}:
	  Compute the Newton correction $(\Delta X,\Delta y,\Delta S)$ from \cref{eq:scea,eq:a,eq:b,eq:c} with right-hand side \cref{eq:rhs_cor}.
	   \State
	   {Update $(X^\ell,y^\ell,S^\ell)$ by
	   $$X^{\ell+1}=X^\ell+\alpha \Delta X,\quad y^{\ell+1}=y^\ell+\beta \Delta y,\quad S^{\ell+1}=S^\ell+\beta\Delta S$$
	   where $\alpha$ and $\beta$ are chosen as in \cref{eq:alphabeta} with $\delta X$, $\delta S$ replaced by $\Delta X$, $\Delta S$, so that $X^{\ell+1}$ and $S^{\ell+1}$ are positive definite.}
	   \EndFor
  \end{algorithmic}
  \end{algorithm}
\section{Preconditioners} \label{sec:precon}
As we mentioned before, the general bottleneck of both algorithms defined above is assembling and solving the Schur complement equation \cref{eq:schurcomp} or the Newton equation \cref{eq:N1}. One of the ways to solve this equation is using an iterative method, in particular, the method of conjugate gradients. 

The Schur complement or Newton equation matrix $H$ is large and sparse, and the assembling of it may be expensive. Besides, and more importantly, it becomes increasing ill-conditioned, in particular in the IP method as it makes progress towards the solution. So a successful CG-based solution of our linear systems relies on an efficient preconditioner. 

In this section, we introduce new efficient preconditioners for this purpose and recall some existing ones. Before doing that, we will introduce a more general form of SDP \cref{eq:SDP} with $p$ matrix variables $X_i\in\S^{m_i}$, $i=1,\ldots p$, and with linear constraints:
\begin{align}
\label{eq:SDPP}
	&\min_{X_1,\ldots,X_p,x_{\text{lin}}~} \sum_{i=1}^p C_i \bullet X_i+d^\top x_{\text{lin}} \\
	&\text{{subject to}}~ \begin{aligned}[t]&\sum_{i=1}^p A_j^{(i)} \bullet X_i + (D^\top x_{\text{lin}})_j =b_j, \quad j=1,\dots,n\\
	&X\succeq 0,\  x_{\text{lin}}\geq 0
\end{aligned}\nonumber
\end{align}
and its dual
\begin{align}
\label{eq:SDPD}
    &\max_{y,S_1,\ldots,S_p,s_{\text{lin}}~} b^\top y&\\
    &\text{{subject to}}~\begin{aligned}[t]&\sum_{j=1}^{n}y_jA_j^{(i)}+S_i=C_i,\quad i=1,\ldots, p\\
    &Dy+s_{\text{lin}}=d\\
    &S_i\succeq 0,\quad i=1,\ldots, p\\
    &s_{\text{lin}}\geq 0\,.
\end{aligned}\nonumber
\end{align}
While, formally, \cref{eq:SDPP} and \cref{eq:SDPD} can be written in the form of \cref{eq:SDP} and \cref{eq:SDD}, respectively, the more detailed form is needed for software implementation and we will thus use it in the following. Furthermore, the formal separation of linear constraints is important for the derivation of the preconditioner.

Accordingly, we have the following new notation: $\BA_i$ is the matrix $\BA$ for the $i$-th LMI, which is
$$\BA_i=[\text{vec}A_1^{(i)},\dots, \text{vec}A_n^{(i)}],$$
and $W_i$ is the scaling matrix $W$ for the $i$-th LMI.

Our goal is to solve efficiently linear systems with matrix $H$, which is either the Schur complement matrix in the IP method or the Hessian in the PDAL method.
In designing the preconditioner, we will assume that $H$ can be written as a sum
$$
  H = \widehat{H} + \sum_{i=1}^p B_i\,.
$$
The idea is to replace the matrices $B_i$ by their approximation $\widetilde{B}_i$, $i=1,\ldots,p$, with the aim to simplify computation of the inverse of $\sum\limits_{i=1}^p\widetilde{B}_i$. We will use the following general theorem.

\begin{theorem}\label{th:precond}
    Let $H = \widehat{H} + \sum\limits_{i=1}^p B_i$ and let $\widetilde{B}_i$, $i=1,\ldots,p$, be matrices given such that $P:=\widehat{H} + \sum\limits_{i=1}^p\widetilde{B}_i$ is positive definite.
    Assume that, for $i=1,\ldots,p$,
    $$
        \lambda_{\rm max}\left(P^{-1/2}(B_i - \widetilde{B}_i) P^{-1/2}\right) \leq \overline\varepsilon_i\
\quad\mbox{and}\quad
        \lambda_{\rm min}\left(P^{-1/2}(B_i - \widetilde{B}_i) P^{-1/2}\right) \geq \underline\varepsilon_i\,.
    $$
    Then 
    $$
        \kappa\left(P^{-1/2} H P^{-1/2}\right) \leq \frac{\displaystyle 1+\sum_{i=1}^p\overline\varepsilon_i}{\displaystyle 1+\sum_{i=1}^p\underline\varepsilon_i}\,.
    $$
\end{theorem}
\begin{proof}
We have 
$$
  H = \sum_{i=1}^{p}B_i + P - \sum_{i=1}^{p}\widetilde{B}_i,
$$
and thus
$$
  P^{-1/2} H P^{-1/2} 
  = P^{-1/2}\left(\sum_{i=1}^{p}(B_i - \widetilde{B}_i)\right) P^{-1/2} + I\,.
$$
Hence
$$
  \lambda_k\left(P^{-1/2} H P^{-1/2} \right) =
  \lambda_k\left(P^{-1/2}\left(\sum_{i=1}^{p}(B_i - \widetilde{B}_i)\right) P^{-1/2}\right) + 1,\quad k=1,\ldots,n\,.
$$

Using the inequalities
$$
     \lambda_{\rm max}\left(P^{-1/2}\left(\sum_{i=1}^{p}(B_i - \widetilde{B}_i)\right)P^{-1/2}\right)
     \leq\sum_{i=1}^p\lambda_{\rm max}\left(P^{-1/2}(B_i - \widetilde{B}_i) P^{-1/2}\right)
$$ 
and
$$
     \lambda_{\rm min}\left(P^{-1/2}\left(\sum_{i=1}^{p}(B_i - \widetilde{B}_i)\right)P^{-1/2}\right)
     \geq\sum_{i=1}^p\lambda_{\rm min}\left(P^{-1/2}(B_i - \widetilde{B}_i) P^{-1/2}\right)
$$ 
we get the required result.
\end{proof}

\subsection{Rank of $H$}
Our goal is to solve efficiently linear systems with matrix $H$, which is either 
the Schur complement matrix in the IP method or the Hessian in the PDAL method. 
In both algorithms, $H$ can be written as a sum
$$
    H = \sum_{i=1}^p H^i_{\text{lmi}} + H_{\text{lin}}\,.
$$

To simplify the notation, \emph{in the following we will only focus on linear systems arising in the IP algorithm; a specification for PDAL will be given in \Cref{sec:prec_AL}}.

\medskip
Recall from \cref{eq:schurcomp} that, in the IP method, every $H^i_{\text{lmi}}$ is computed as
\begin{equation}\label{eq:H}
    H^i_{\text{lmi}} = \BAT_i (W_i\otimes W_i)\BA_i\,,
\end{equation}
where $W_i$ is the NT scaling matrix. 

In the following \cref{th:rankW}--\labelcref{th:l2}, we will omit the subscript $i$ and refer to the general problem \cref{eq:SDP} before returning to the notation of this section in \cref{th:rara}.
The next lemma shows that $W$ has the same rank as $X$. So, as $X$ is low-rank, $W$ will be low-rank. Later in this section, we will use this low-rank property of $W$ to present the preconditioners.

Let $(X,S)$ be a pair of solutions to \cref{eq:SDP}. By complementarity, we know that $XS=0$, and $\text{rank}(X)+\text{rank}(S)=k+\sigma\leq m$. Assume $k+\sigma=m$, i.e., strict complementarity. From $XS=0, X\succeq 0, S\succeq 0$, we know that $X$ and $S$ are simultaneously diagonalizable, i.e., there exists $U$ such that $X=U\Lambda_X U^\top $, $S=U\Lambda_S U^\top $, where $\Lambda_X$ and $\Lambda_S$ are the diagonal matrices of eigenvalues and $U^\top U=I$.
\begin{lemma}\label{th:rankW}
	Let $X,S$ be as above. Let $W\in\mathbb{S}^m$ be any matrix such that $X=WSW$. Then $\text{rank}(W)=\text{rank}(X)$.
\end{lemma}
\begin{proof}
	Assume, without loss of generality, that eigenvalues of $X$ and $S$ are sorted such that
	$\Lambda_X={\rm Diag} \left(\lambda_1,\ldots,\lambda_k,0,\ldots,0\right)$ and 
	$\Lambda_S={\rm Diag} \left(0,\ldots,0,\mu_1,\ldots,\mu_{m-k}\right)$.
Then $X=WSW$ is equivalent to $U\Lambda_XU^\top =WU\Lambda_SU^\top W$, and so, as $U^\top U=I$, $\Lambda_X=U^\top WU\Lambda_SU^\top WU$. Define $Z=U^\top WU$, so that $\Lambda_X = Z^\top \Lambda_S Z$. Because $(\Lambda_X)_{i,i}=0$ for $i>k$, it must hold that $z_{:,i}^\top z_{:,i}=0$, $i>k$, where $z_{:,i}$ is the $i$-th column of $Z$. As $Z$ is positive semidefinite, it means that $Z_{i,j}=0$ for $i,j>k$, i.e., the leading $k\times k$ submatrix of $Z$ is a rank-$k$ non-zero matrix and the rest of the matrix is zero. Hence, $\rank(Z)=k$ and, as $W=UZU^\top $, so is the rank of $W$.
\end{proof}

The proof of the next lemma is straightforward.
\begin{lemma}\label{th:l1}
Let $X\in \S^m$ such that $\rank X = k$, $k\leq m$. Then $\rank (X\otimes X) = k^2$.
\end{lemma}
\begin{lemma}\label{th:l2}
Let $Y\in \S^m$ such that $\rank Y = k$, $k\leq m$, and $A\in\RR^{n\times m}$, $n<m$. Then $\rank(AYA^\top )\leq k$.
\end{lemma}
\begin{proof}
Because $\rank Y = k$, we have $Y=\sum_{i=1}^k y^{(i)}(y^{(i)})^\top $ with some $y^{(i)}\in\RR^m$, $i=1,\ldots,k$. Hence $AYA^\top  = A\left(\sum_{i=1}^k y^{(i)}(y^{(i)})^\top \right)A^\top  = \sum_{i=1}^k z^{(i)}(z^{(i)})^\top $ with $z^{(i)}=Ay^{(i)}$. Therefore $\rank (AYA^\top )=\rank\sum_{i=1}^k z^{(i)}(z^{(i)})^\top  \leq k$. A strict inequality occurs, trivially, when $n<k$. It can also occur when $z^{(i)}$, $i=1,\ldots,k$, are linearly dependent.
\end{proof}

By combining the above results we get the following theorem.
\begin{theorem}\label{th:rara}
Let, for some $i\in\{1,\ldots,p\}$, $W_i$ be the scaling matrix from \Cref{sec:IP} and let $\rank W_i = k$. Let further $H^i_{{\rm lmi}}$ be defined as in \cref{eq:H}. Then $\rank H^i_{\rm{lmi}} \leq\ k^2$.
\end{theorem}
\begin{remark}
All results in this section have been presented under the assumption of exact rank of the involved matrices. This is, of course, only the limit case in our algorithms. The actual matrices have outlying eigenvalues and converge to the low-rank matrices. The message of \cref{th:rara} remains the same, though: when $W_i$ has $k$ outlying eigenvalues, then $H^i_{\rm{lmi}}$ will have at most $k^2$ outlying eigenvalues. And this, only, is the fact on which we build the preconditioner.
\end{remark}
\subsection{$H_\alpha$ preconditioner}\label{sec:Halpha}
The preconditioner introduced in this section is motivated by the work of Zhang and Lavaei \cite{Zhang_2017} and, in its derivation, we partly follow their paper; however, the new preconditioner differs in a substantial detail, as explained below. Also, our assumption about the matrix $W$ is much weaker.

In \Cref{sec:IP}, we introduced the scaling matrix $W$ and different ways of constructing it. This matrix becomes progressively ill-conditioned as the IP method approaches the solution. This ill-conditioning of $W$ is a result of \cref{th:rankW}. Assuming that the solution matrices $X_i^*$, $i=1,\ldots,p$, have ranks $k_i$, the rank of matrices $W_i$ will also tend to $k_i$. 

The main idea of the preconditioner is to decompose matrices $W_i$ accordingly to their expected rank as follows:
\begin{align}
\label{eq:w0}
W_i=W_i^0+U_iU_i^\top ,
\end{align}
where $U_i$ are $m\times k_i$ matrices of full column rank. Assume that the eigenvalues of $W_i$ are sorted as follows
$$
    \lambda_1(W_i)\leq\cdots\leq\lambda_{m-k_i}(W_i)\ll \lambda_{m-k_i+1}(W_i)\leq\cdots\leq\lambda_m(W_i)\,.
$$
Decomposition~\cref{eq:w0} can be readily computed by spectral decomposition of $W_i$:
\begin{align*}
W_i=\begin{bmatrix}
V_i^s &V_i^l\\
\end{bmatrix}
\begin{bmatrix}
\Lambda_i^s & 0 \\
0 & \Lambda_i^l\\
\end{bmatrix}
\begin{bmatrix}
V_i^s &V_i^l\\
\end{bmatrix}^\top ,
\end{align*}
where $\Lambda_i^s = {\rm diag}(\lambda_1(W_i),\ldots,\lambda_{m-k_i}(W_i))$ and $\Lambda_i^l = {\rm diag}(\lambda_{m-k_i+1}(W_i),\ldots,\lambda_m(W_i))$. Then, by choosing $\tau_i$ such that $\lambda_{1}(W_i)\leq \tau_i <\lambda_{m-k_i}(W_i)$, we obtain the following form of~\cref{eq:w0}:
\begin{equation}
\label{eq:wsplit}
W_i= \underbrace{\begin{bmatrix}
	V_i^s &V_i^l\\
	\end{bmatrix}
	\begin{bmatrix}
	\Lambda_i^s & 0 \\
	0 & \tau_i I\\
	\end{bmatrix}
	\begin{bmatrix}
	V_i^s &V_i^l\\
	\end{bmatrix}^\top }_{W_i^0}
+\underbrace{ V_i^l(\Lambda _i^l - \tau I)(V_i^l)^\top }_{U_iU_i^\top }.
\end{equation}

Recall now the form of the Schur complement matrix for problem~\cref{eq:SDPP}:
\begin{align}
\label{Hmorelmi}
H=\sum_{i=1}^{p}\BAT_i (W_i\otimes W_i)\BA_i+D^\top X_{\text{lin}}S_{\text{lin}}^{-1}D,
\end{align}
in which
$X_{\text{lin}}=\text{diag}(x_{\text{lin}})$ and $S_{\text{lin}}=\text{diag}(s_{\text{lin}})$. 

\begin{remark}{
Formally, we could treat the linear constraints (if present) as matrix constraints with diagonal matrices and perform the above decomposition for these constraints, too. However, it is rather unlikely that the corresponding part of the solution, vector $x_{\rm lin}$ would be ``low rank'', i.e., that only a very few of the linear constraints would be active. This fact could then ruin the whole idea. Here we propose to treat the linear constraints as ``full rank'' and add the matrix $D^\top X_{\text{lin}}S_{\text{lin}}^{-1}D$ to the preconditioner {as is}. In our application, the linear constraints only consist of upper and lower bounds (in the dual formulation), hence the matrix $D^\top X_{\text{lin}}S_{\text{lin}}^{-1}D$ is diagonal and satisfies \cref{assum:a4}.
}
\end{remark}
By substituting the splittings \cref{eq:wsplit} for $i=1,\dots,p$ into the matrix \cref{Hmorelmi}, we get
\begin{align*}
H&=\sum_{i=1}^{p}\BA_i^\top \left((W_i^0+U_iU_i^\top )\otimes (W_i^0+U_iU_i^\top )\right)\BA_i+D^\top X_{\text{lin}}S_{\text{lin}}^{-1}D\\
\begin{split}=\sum_{i=1}^{p}\BA_i^\top \left(W_i^0\otimes W_i^0+U_iU_i^\top \otimes W_i^0+W_i^0\otimes U_iU_i^\top +U_iU_i^\top \otimes U_iU_i^\top \right)\BA_i\\+D^\top X_{\text{lin}}S_{\text{lin}}^{-1}D.\end{split}
\end{align*}
Using the identity $\BAT (\Phi\otimes \Xi)\BA=\BAT (\Xi\otimes \Phi)\BA$ for any $\Phi,\Xi\in\RR^{m\times m}$ with $\BA$ defined as above (\cite[Lemma 6]{Zhang_2017} and \cite[Chap.\,4.2, Problem 25]{horn1991topics}), we obtain
\begin{align*}
H=\sum_{i=1}^{p}\BA_i^\top (W_i^0\otimes W_i^0)\BA_i+\sum_{i=1}^{p}\BA_i^\top (U_i\otimes \Gamma_i)(U_i\otimes \Gamma_i)^\top \BA_i+D^\top X_{\text{lin}}S_{\text{lin}}^{-1}D,
\end{align*}
where $\Gamma_i$ is any matrix satisfying $\Gamma_i\Gamma_i^\top =2W_i^0+U_iU_i^\top $. Next we define $V_i=\BAT_i (U_i\otimes \Gamma_i)$ for $i=1,\dots,p$. Then,
\begin{align}
\label{modifiedH}
H&=\sum_{i=1}^{p}\BAT_i (W_i^0\otimes W_i^0)\BA_i+D^\top X_{\text{lin}}S_{\text{lin}}^{-1}D+\sum_{i=1}^{p}V_iV_i^\top &\nonumber\\
&=\sum_{i=1}^{p}\underbrace{\BAT_i (W_i^0\otimes W_i^0)\BA_i}_{H_i^0}+D^\top X_{\text{lin}}S_{\text{lin}}^{-1}D+\tilde{V}\tilde{V}^\top \,,&
\end{align}
where $\tilde{V}=[V_1,\dots, V_p]$.

If we approximate $\BAT_i (W_i^0\otimes W_i^0)\BA_i$ by $\tau_i^2 I$ in \cref{modifiedH}, the result yields to ${H}_\alpha$ preconditioner, which is
\begin{align}
\label{halph}
{H}_\alpha=\underbrace{\left(\sum_{i=1}^{p}\tau_i^2I+D^\top X_{\text{lin}}S_{\text{lin}}^{-1}D\right)}_{{\BA}_\alpha}+\tilde{V}\tilde{V}^\top ,
\end{align}
or, in short,
\begin{align}
\label{Halpha}
{H}_\alpha={\BA}_\alpha+\tilde{V}\tilde{V}^\top\,.
\end{align}
This is the first preconditioner that we introduce. By using the Sherman-Morrison-Woodbury (SMW) formula, its inverse will be
\begin{align}
\label{halphsmw}
{H}_\alpha^{-1}={\BA}_\alpha^{-1}(I-\tilde{V}\Theta^{-1}\tilde{V}^\top  {\BA}_\alpha^{-1})\,,
\end{align}
where $\Theta= I+\tilde{V}^\top {\BA}_\alpha^{-1}\tilde{V}$ is the Schur complement. Notice that, by \cref{assum:a4}, the inverse of ${\BA}_\alpha$ is inexpensive.


\begin{theorem}\label{th:precond1}
    Assume that, for $i=1,\ldots,p$,
    $$
        \lambda_{\rm max}\left(H_\alpha^{-1/2}(H^0_i - \tau_i^2I) H_\alpha^{-1/2}\right) \leq \overline\varepsilon_i\
\quad\mbox{and}\quad
        \lambda_{\rm min}\left(H_\alpha^{-1/2}(H^0_i - \tau_i^2I) H_\alpha^{-1/2}\right) \geq \underline\varepsilon_i\,.
    $$
    Then 
    $$
        \kappa\left(H_\alpha^{-1/2} H H_\alpha^{-1/2}\right) \leq \frac{\displaystyle 1+\sum_{i=1}^p\overline\varepsilon_i}{\displaystyle 1+\sum_{i=1}^p\underline\varepsilon_i}\,.
    $$
\end{theorem}
\begin{proof}
The result follows directly from \cref{th:precond} by setting $\widehat{H}=D^\top X_{\text{lin}}S_{\text{lin}}^{-1}D+\tilde{V}\tilde{V}^\top$, $B_i = \BAT_i (W_i^0\otimes W_i^0)\BA_i$ and $\widetilde{B}_i = \tau_i^2I$, $i=1,\ldots,p$.

\end{proof}

\begin{algorithm}
\caption{Solution of the linear system $H\Delta y=r$ by PCG}
\label{algo:PCG}
Given data $k>0$ (solution rank), $W_i\in \mathbb{S}_{++}^m$, $i=1,\ldots,p$, $r\in \mathbb{R}^n$ (right-hand side) an and initial iterate~$\Delta y^0$.
\begin{algorithmic}[1]
\Procedure{Setting up preconditioner $H_\alpha$}{}
\For{$i=1,\ldots,p$}\Comment{computing the decomposition \cref{modifiedH}}
\State
		Compute eigenvalue decomposition $W_i=V_i\Lambda_i V_i^\top $ and
		set $\tau_i=\lambda_\text{min}(W_i)$.
\State
		Form the matrices $W_i^0$ and $U_i$ by \cref{eq:wsplit}.
		\State Compute
		the Cholesky factorization $\Gamma_i\Gamma_i^\top =2W_i^0+U_iU_i^\top $.
		\EndFor
\State Compute the Schur complement
		$\Theta= I+\tilde{V}^\top {\BA}_\alpha^{-1}\tilde{V}$ and
		its Cholesky factorization $\Theta = LL^\top$.
\EndProcedure
\Procedure{PCG}{}
\Statex Use standard PCG algorithm to solve $H\Delta y = r$ until 
$ \|H\Delta y - r\|/\|r\| \leq \varepsilon_{\scriptscriptstyle \rm CG}$.
\Statex At each PCG iteration:
\State Compute the matrix-vector product with $H$ using \cref{eq:linearsys}.
\State Compute the matrix-vector product with $H_\alpha^{-1}$
		by means of the SMW in \cref{halphsmw},
		using $\Theta^{-1}={L}^{-\!\top}{L}^{-1}$.
\EndProcedure
\end{algorithmic}
\end{algorithm}

\paragraph{Complexity of the PCG}
\begin{itemize}
\item We compute the Schur complement
	$$\Theta= I+[\BA_1^\top (U_1\otimes \Gamma_1) \dots \BA_p^\top (U_1\otimes \Gamma_p)]^\top {\BA}_\alpha^{-1}[\BA_1^\top (U_1\otimes \Gamma_1) \dots \BA_p^\top (U_p\otimes \Gamma_p)],$$
block-wise, with the $(i,j)$-block $\Theta_{ij} = B_i^\top B_j$, where $B_i = R^{-1} \BA_i^\top (U_i\otimes \Gamma_i)$, $i=1,\ldots,p$, and $R$ is the Cholesky factor of $\BA_\alpha$. Then, we multiply $(C_i=)R^{-1}\BA_i^\top$, which is $\mathcal{O}(n)$ flops. The sparsity of $\BA_i$ is maintained in $C_i$, and so $C_i d$ is still $\mathcal{O}(n)$ flops for an arbitrary vector $d$. Hence, as $(U_i\otimes \Gamma_i)$ has $m_ik_i$ columns, the product $C_i(U_i\otimes \Gamma_i)$ requires $\mathcal{O}(nm_i k_i)$ flops. Finally, we compute $B_i^\top B_j$, which requires $\mathcal{O}(m_i m_j k_i k_j)$ flops. Define $\hat{m}=\max\limits_{i=1,\ldots,p}m_i$ and $\hat{k}=\max\limits_{i=1,\ldots,p}k_i$. 
Forming and factorizing the Schur complement (i.e., forming the preconditioner) requires
$\mathcal{O}\left(p^2\hat{m}^2 \hat{k}^2 + pn\hat{m}\hat{k} + \hat{m}^3 \hat{k}^3\right)$ flops; the last term comes from the Cholesky factorization of~$S$.

\item Each iteration of the PCG method is dominated by the matrix-matrix product in \cref{eq:linearsys} requiring $\mathcal{O}\left(\sum_{i=1}^p m_i^3\right)$ flops and application of the preconditioner \cref{halphsmw} requiring $\mathcal{O}\left(\sum_{i=1}^p m_i k_i\right)^2 + \mathcal{O} \left(n\sum_{i=1}^p m_ik_i\right)$ flops.

\item
The number of PCG iterations is a little unpredictable, in particular, when approaching the solution of the problem. However, our experience shows that this number is usually well below 10 and does not exceed 100 in the worst problems. 
Dropping the lower-order terms and assuming $n\approx \hat{m}^2$, we will need $\mathcal{O}\left(\hat{m}^3\hat{k}^3\right)$ flops in the preconditioner and $\mathcal{O}\left(\hat{m}^3\hat{k}\right)$ flops in one PCG iteration. 
We come to the conclusion that the cost of the preconditioner is about the same as the cost of the solution of the linear system by the PCG method.
\end{itemize}

\paragraph{$\widetilde{H}$ preconditioner of \cite{Zhang_2017}}
The main idea of proposing preconditioner $H_\alpha$ comes from 
Zhang and Lavaei \cite{Zhang_2017}. In their paper, they do not consider linear constraints and approximate $\BA_i^{\!\!\top} (W_i^0\otimes W_i^0)\BA_i$ by $\tau_i^2 \BA_i^{\!\!\top}  \BA_i$ in \cref{modifiedH}, the result yielding the following preconditioner denoted in \cite{Zhang_2017} by $\tilde{H}$ (and adopted here to the case of more LMIs and additional linear constraints):
\begin{align}
\label{htil}
\widetilde{H}=\sum_{i=1}^{p}\tau_i^2\BA_i^{\!\!\top} \BA_i+D^\top X_{\text{lin}}S_{\text{lin}}^{-1}D+\tilde{V}\tilde{V}^\top \,.
\end{align}
Obviously, the only difference between $H_\alpha$ and $\widetilde{H}$ is in the first term of the sum. As this term needs to be inverted in the SMW formula, $\widetilde{H}$ requires an additional assumption that it is easily invertible. This assumption is not satisfied in our problems. Later in \Cref{sec:examples}, \cref{tab:other}, we will demonstrate that, despite being seemingly simple, $H_\alpha$ is at least as efficient as $\widetilde{H}$ and requires much less computational time.
\subsection{$H_\beta$ preconditioner}
Recall the definition of the matrix $H$:
$$
    H = \sum_{i=1}^{p}\BAT_i (W_i^0\otimes W_i^0)\BA_i+\tilde{V}\tilde{V}^\top+D^\top X_{\text{lin}}S_{\text{lin}}^{-1}D\,.
$$
It turns out that, in our numerical examples, the last term is dominating in the first iterations of IP, before the low-rank structure of $W$ is clearly recognized. This is demonstrated in \cref{fig:prec1} that shows distribution of eigenvalues of matrices $D^\top X_{\text{lin}}S_{\text{lin}}^{-1}D$ and $\sum_{i=1}^{p}\BAT_i (W_i^0\otimes W_i^0)\BA_i+ \tilde{V}\tilde{V}^\top$ in iterations 2, 17, 30 (final) in problem tru7e (see \Cref{sec:examples}).
	\begin{figure}[h]
	\begin{center}
		\resizebox{0.32\hsize}{!}
		{\includegraphics{{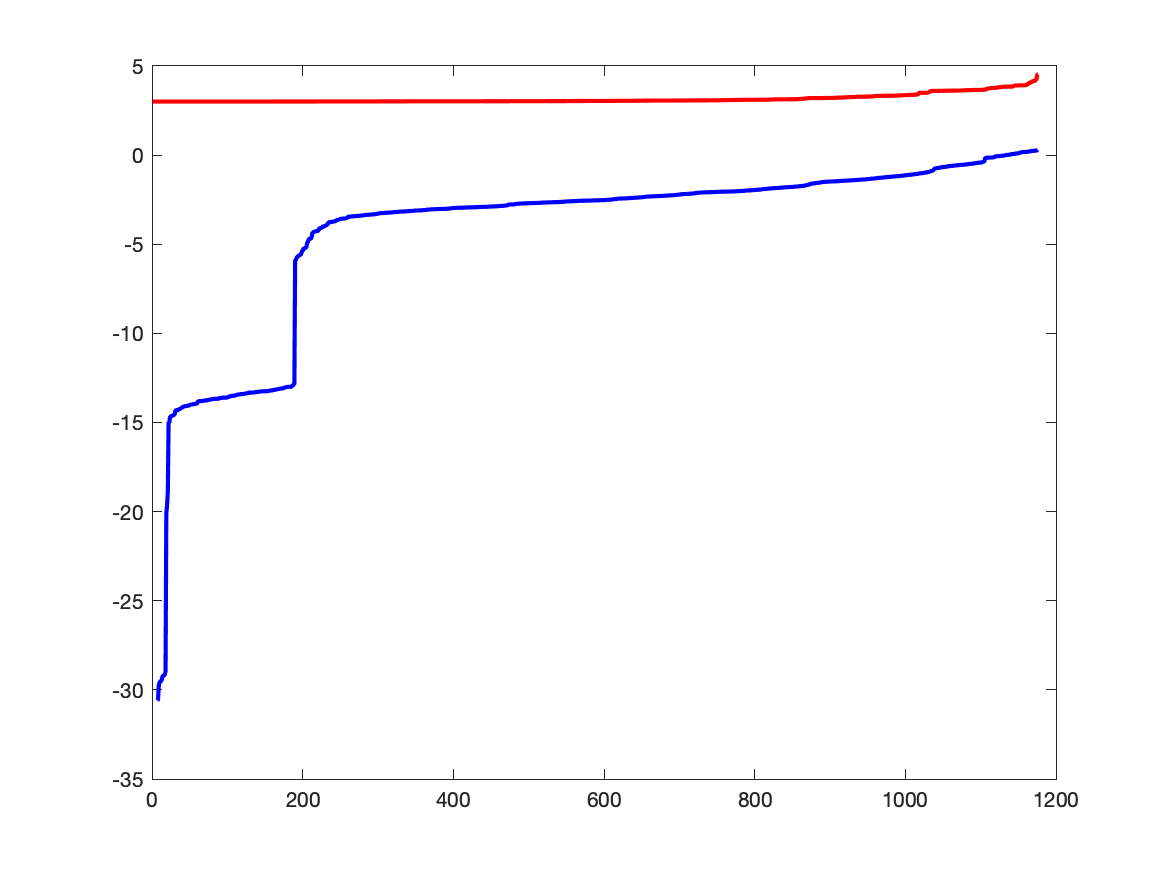}}}
		\resizebox{0.32\hsize}{!}
		{\includegraphics{{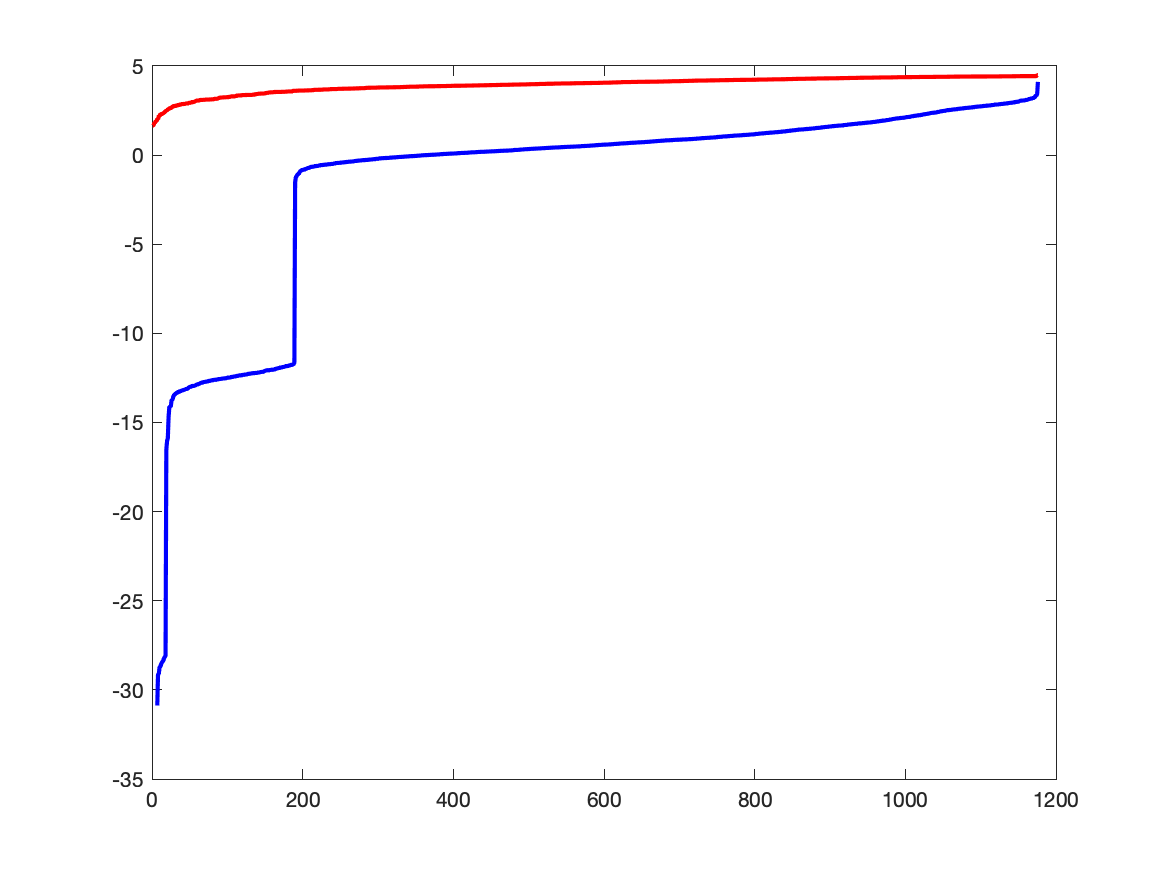}}}
		\resizebox{0.32\hsize}{!}
		{\includegraphics{{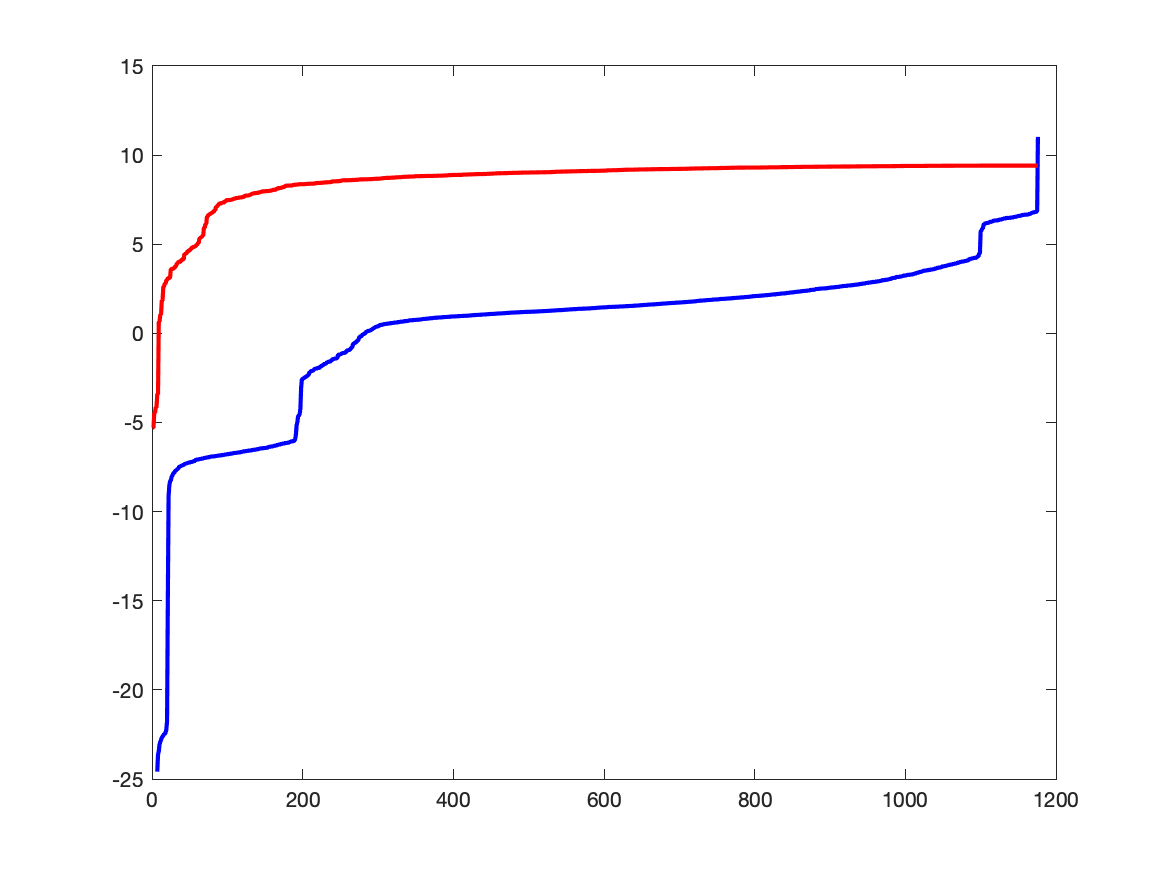}}}
	\end{center}
	\caption{Problem tru7e; eigenvalues of $D^\top X_{\text{lin}}S_{\text{lin}}^{-1}D$ (red) and $\sum_{i=1}^{p}\BA_i^\top (W_i^0\otimes W_i^0)\BA_i+ \tilde{V}\tilde{V}^\top$ (blue) in iteration 2, 17 and 30 (left to right).}
	\label{fig:prec1}
	\end{figure}
This observation lead to the idea of a simplified preconditioner called $H_\beta$ and defined as follows
\begin{align}
\label{eq:Hbeta}
H_\beta=\sum_{i=1}^{p}\tau_i^2I+D^\top X_{\text{lin}}S_{\text{lin}}^{-1}D,
\end{align}
in which $\tau_i$ is defined as in the previous section.
This matrix is easy to invert by \cref{assum:a4}; in fact, the matrix is diagonal in problems in \Cref{sec:examples}. It is therefore an extremely ``cheap'' preconditioner that is efficient in the first iterations of the IP algorithm.

\subsection{Preconditioner used in augmented Lagrangian approach}\label{sec:prec_AL}

\medskip

As has been shown previously in \Cref{Penalty/Barrier}, see \cref{eq:Hessianaug} and \cref{eq:N1}, the system matrix in the Newton systems in both, the purely primal and the primal-dual approach, is computed by the formula
\begin{equation} \label{eq:H_AL}
    H=2\sum_{i=1}^p \BAT_i (W_i\otimes V_i)\BA_i + H_{\rm{lin}},
\end{equation}
where 
\begin{align}
    H_{\rm{lin}}:&=rI + D^\top \bar{W}^{\rm{lin}}\left(y,X_i^{\rm{lin}},\pi^{\rm{lin}}\right) D ,\label{eq:Hlin}\\
    W_i:&=\frac{1}{\pi^{\rm{lmi}}}\bar{X}_i(y)=\pi^{\rm{lmi}}\mathcal{Z}_i(y)X_i^{\rm{lmi}}\mathcal{Z}_i(y),\\
    V_i:&=\pi^{\rm{lmi}}\mathcal{Z}_i(y)=-\pi^{\rm{lmi}}\left(A_i^{\rm{lmi}}(y)-\pi^{\rm{lmi}}I_i\right)^{-1}.
\end{align}
Here the lower index $i$ indicates the $i$-th LMI constraint and $A_i^{\rm{lmi}}: \mathbb{R}^n \to \mathbb{S}^{m_i}, y \mapsto \sum_{j=1}^n y_j A_i^{(j)} - C_i$. Moreover, $y$ is the current primal iterate, $X_i \in \mathbb{S}^{m_i}$, $i=1,\ldots,p$, are the matrix multipliers from the current outer iteration, $\pi^{\rm{lmi}},\pi^{\rm{lin}}$ are the current penalty parameters and $\bar{W}^{\rm{lin}}\left(y,X_i^{\rm{lin}},\pi^{\rm{lin}}\right)$ is the diagonal matrix function defined in \cref{eq:Wlin}. 

As the matrices $W_i$, $i=1,\ldots,p$, are scaled versions of the matrix multipliers, the low rank \Cref{assum:alr} directly applies to those. In addition, the eigenvalue structure of the matrices $V_i$, $i=1,\ldots,p$, is characterized by the following lemma.
\begin{lemma}\label{th:eig_Vi}
\begin{enumerate}
\item[a)] Let $y\in \mathbb{R}^n$ be feasible with respect to all matrix inequalities. Then $0 < \lambda_j(V_i) \leq 1$ for all $j=1,\ldots,m_i$ and all $i=1,\ldots,p$.
\item[b)] If strict complementarity holds in a primal solution $y^*$ with associated matrix multipliers $X^*_1,\ldots,X^*_p$ and the penalty parameter $\pi^{\rm{lmi}}$ is chosen sufficiently small, then the number of outlying eigenvalues of $V_i$ is (in a neighbourhood of the solution $y^*$) given by the rank $k_i$ of the corresponding multiplier matrix $X_i^*$ for all $i=1,\ldots,p$. 
\end{enumerate}
\end{lemma}
\begin{proof}
\begin{enumerate}
\item[a)] Writing $V_i$ as a primary matrix function, we observe that the eigenvalues of every $A_i^{\rm{lmi}}(y)$ are transformed by the nonlinear function $\psi: \mathbb{R} \to \mathbb{R}, t \to -\pi^{\rm{lmi}}/(t-\pi^{\rm{lmi}})$. Assuming that $y$ is feasible with respect to all matrix constraints, we have that all eigenvalues are non-positive. Now, taking the monotonicity of $\psi$, $\psi \geq 0$ on $(-\infty,\pi^{\rm{lmi}})$, $\psi(t) \to 0$ for $t \to -\infty$ and $\psi(0)=1$ into account, we see that $0 < \lambda_j(V_i) \leq 1$ for all $j=1,\ldots,m_i$ and all $i=1,\ldots,p$.
 \item[b)] Let $\lambda_i^1 \leq \ldots \leq \lambda_i^{m_i-k_i} < \lambda_i^{m_i-k_i+1} = \lambda_i^{m_i-k_i+2} = \ldots = \lambda_i^{m_i} = 0$ denote the eigenvalues of $A_i^{\rm{lmi}}(y^*)$ where $k_i$ denotes the rank of the associated multiplier matrix $X^*_i$. By, assumption {$\pi^{\rm{lmi}}$ is sufficiently small}, so we have that $\lambda_i^1 \leq \ldots \leq \lambda_i^{m_i-k_i} \ll -\pi^{\rm{lmi}} < 0$ and $0 < -\pi^{\rm{lmi}} / \lambda_i^{m_i-k_i} \ll 1$. Finally, by the nonlinear transformation $\psi$ of the eigenvalues of $V_i$, we see that $0 < \psi(\lambda_i^{j}) \ll 1$ for all $j=1,\ldots,m_i-k_i$ and $\psi(\lambda_i^{j})=1$ for all $j=m_i-k_i+1,\ldots,m_i$ and all $i=1,\ldots,p$.
\end{enumerate}
\end{proof}


\cref{th:eig_Vi} suggests how to construct preconditioners for two different situations: in the first situation, we assume for a given LMI with index $i$ that only the matrix $W_i$ has a distinct eigenvalue structure, while for $V_i$ we just assume that the eigenvalues are essentially bounded between $0$ and $1$. In the second situation, we have that both matrices $V_i$ and $W_i$ exhibit the same eigenvalue structure. We note that this is somehow similar to the situation in the interior-point setting; however there, the matrices $V_i$ and $W_i$ would be the same. 

We would like to note that in the numerical examples presented in \Cref{sec:examples}, we will be always in the first situation, which means that the eigenvalues of $V_i$ will never be clearly separated from $1$. Nevertheless, we would like to demonstrate in the following how to construct efficient preconditioners for both situations.


For this, we can use the same idea as has been presented for the preconditioner $H_\alpha$ in \Cref{sec:Halpha}.
The first step is to decompose the matrix $H$ in \cref{eq:H_AL}. We introduce the following decompositions for $V_i$ and $W_i$, $i=1,\ldots,p$, which separate small from outlying eigenvalues precisely in the same way as in \cref{eq:wsplit}:
\begin{align}
    V_i&=V_i^0+\Tilde{V_i}\Tilde{V_i}^\top , \label{vdecom}\\
    W_i&=W_i^0+\Tilde{W_i}\Tilde{W_i}^\top. \label{wdecom}
\end{align}


Based on this we can formulate the following auxiliary results:
\begin{lemma}\label{th:H_AL_Dec2}
    $H$ can be written as
    \begin{equation}\label{eq:H_AL_Dec2}
        H=H_{\rm{lin}}+H^0_\gamma+H^{\rm{lr}}_\gamma,
    \end{equation}
    where
    \begin{equation*}
    \begin{aligned}
    H^0_\gamma&=2\sum_{i=1}^p \BAT_i(W_i^0\otimes V_i)\BA_i, \nonumber\\
        H^{\rm{lr}}_\gamma&=2\sum_{i=1}^p \BAT_i (\Tilde{W}_i\otimes \Delta_i)(\Tilde{W}_i\otimes \Delta_i)^\top \BA _i,
         \end{aligned}
    \end{equation*}
    and $\Delta_i$ is chosen such that $V_i=\Delta_i \Delta_i^\top $ for all $i=1,\ldots,p$. 
\end{lemma}
\begin{proof}
It holds that
 \begin{equation*}
     \begin{split}
     H&=H_{\rm{lin}}+2\sum_{i=1}^p\BAT_i(W_i\otimes V_i)\BA_i,\\
        &=H_{\rm{lin}}+2\sum_{i=1}^p\BAT_i(W_i^0\otimes V_i)\BA_i+2\sum_{i=1}^p\BAT_i(\Tilde{W}_i\Tilde{W}_i^\top \otimes V_i)\BA_i,\\
        &=H_{\rm{lin}}+2\sum_{i=1}^p\BAT_i(W_i^0\otimes V_i)\BA_i+2\sum_{i=1}^p\BAT_i(\Tilde{W}_i\Tilde{W}_i^\top \otimes \Delta_i \Delta_i^\top)\BA_i\,.
     \end{split}
 \end{equation*}
Now, \cref{eq:H_AL_Dec2} follows using standard properties of the Kronecker product.
\end{proof}

\begin{lemma}\label{th:H_AL_Dec1}
    $H$ can be written as
    \begin{equation}\label{eq:H_AL_Dec1}
        H=H_{\rm{lin}}+H^0_\delta+H^{\rm{lr}}_\delta,
    \end{equation}
    where
    $$\begin{aligned}
    H^0_\delta&=2\sum_{i=1}^p \BAT_i(W_i^0\otimes V_i^0)\BA_i, \nonumber\\
    H^{\rm{lr}}_\delta&=2\sum_{i=1}^p \BAT_i\left((\Tilde{W}_i\otimes \Theta_i)(\Tilde{W}_i\otimes \Theta_i)^\top + (\Tilde{V}_i\otimes \Gamma_i)(\Tilde{V}_i\otimes \Gamma_i)^\top\right)\BA_i\,,
    \nonumber
    \end{aligned}$$
    and the matrices $\Gamma_i$ and $\Theta_i$ are chosen such that $\Gamma_i\Gamma_i^\top =W_i^0+\frac{1}{2}\Tilde{W}_i\Tilde{W}_i^\top $ and  $\Theta_i \Theta_i^\top =V_i^0+\frac{1}{2}\Tilde{V}_i\Tilde{V}_i^\top $ for all $i=1,\ldots,p$.
\end{lemma}
\begin{proof}
 Using \cref{vdecom} and \cref{wdecom} and the definitions of $\Gamma_i, \Theta_i$, $i=1,\ldots,p$, we see: 
\begin{equation*}
    \begin{split}
        H&=H_{\rm{lin}}+2\sum_{i=1}^p\BAT_i(W_i\otimes V_i)\BA_i,\\
        &=H_{\rm{lin}}+2\sum_{i=1}^p\BAT_i\left((W_i^0+\Tilde{W}_i\Tilde{W}_i^\top ) \otimes (V_i^0+\Tilde{V}_i\Tilde{V}_i^\top ) \right)\BA_i\\
        &=H_{\rm{lin}}+2\sum_{i=1}^p\BAT_i(W_i^0\otimes V_i^0)\BA_i\\
        &\qquad +2\sum_{i=1}^p\BAT_i\left((V_i^0+\frac{1}{2}\Tilde{V}_i\Tilde{V}_i^\top )\otimes\Tilde{W}_i\Tilde{W}_i^\top +(W_i^0+\frac{1}{2}\Tilde{W}_i\Tilde{W}_i^\top )\otimes\Tilde{V}_i\tilde{V}_i^\top \right)\BA_i \\
        &=H_{\rm{lin}}+2\sum_{i=1}^p\BAT_i(W_i^0\otimes V_i^0)\BA_i\\
        &\qquad +2\sum_{i=1}^p\BAT_i\left(\Tilde{W}_i\Tilde{W}_i^\top \otimes\Theta_i \Theta_i^\top +\Tilde{V}_i\tilde{V}_i^\top \otimes \Gamma_i \Gamma_i^\top \right)\BA_i\,.
    \end{split}
\end{equation*}
Now, again, \cref{eq:H_AL_Dec1} follows using standard properties of the Kronecker product.
\end{proof}

Finally, efficient preconditioners can be built on the basis of the decompositions \cref{eq:H_AL_Dec2} and \cref{eq:H_AL_Dec1} by approximating the $H^0_\delta$ and $H^0_\gamma$ -terms as follows:
 \begin{equation*}
     \begin{split}
     H_\gamma&=H_{\rm{lin}}+\sum_{i=1}^p\tau_i^{(1)}\tau_i^{(2)} {\rm diag}\left(\BA_{i} \BAT_i\right)+H^{\rm{lr}}_\delta,\\
     H_\delta&=H_{\rm{lin}}+\sum_{i=1}^p\tau_i^{(1)}\tau_i^{(2)} {\rm diag}\left(\BA_{i} \BAT_i\right)+H^{\rm{lr}}_\gamma.
     \end{split}
 \end{equation*}
Here, for all $i=1,\ldots,p$, we choose $\tau_i^{(1)}$ to be the smallest eigenvalue of $W_i^0$ multiplied by 10 and $\tau_i^{(2)}$ to be the average of the eigenvalues of $V_i^0$ or $V_i$, respectively. The expression ${\rm diag}\left(\BA_{i} \BAT_i\right)$ denotes a diagonal matrix of size $m_i\times m_i$ computed from $\BA_{i} \BAT_i$.

We finally note that the terms $H_{\rm{lin}}$ and $H^{\rm{lr}}_\gamma$, $H^{\rm{lr}}_\delta$ can be treated using the SMW-formula exactly as has been demonstrated in \cref{halphsmw}. Moreover, thanks to \cref{th:H_AL_Dec2} and \cref{th:H_AL_Dec1}, \cref{th:precond} applies to the preconditioners presented above and also the complexity formulas are the same as presented for $H_{\alpha}$ preconditioner due to their analogous construction.


\section{Application: truss topology optimization}\label{sec:truss}
\subsection{Truss notation}
\emph{The notation used throughout \Cref{sec:truss} is specific for this application and unrelated to the rest of the paper.}

By truss we understand a mechanical structure, an assemblage of pin-jointed uniform straight
bars made of elastic material, such as steel or aluminium. The bars can only carry axial tension and compression. We denote by $m$ the number of bars and
by $N$ the number of joints. The positions of the joints are collected
in a vector $y$ of dimension $\tn:=dim\cdot N$ where $dim$ is the
spatial dimension of the truss.
The material properties of bars are characterized by their Young's
moduli $E_i$, the  bar lengths are denoted by $\ell_i$ and bar
volumes by $t_i$, $i=1,\ldots,m$.

For an $i$-th bar, let $\gamma_i$ be the vector of directional cosines; e.g., for $dim=2$, we define
$$
    \gamma_i = {\frac{1}{\ell_i}}\left(
    -(y_1^{(k)}-y_1^{(j)}),\quad -(y_2^{(k)}-y_2^{(j)}),\quad
    y_1^{(k)}-y_1^{(j)},\quad
    y_2^{(k)}-y_2^{(j)}
    \right)^\top ,
$$
where $y^{(j)}$ and $y^{(k)}$ are the end-points of the bar.

Let $f\in\RR^{\tn}$ be a load vector of nodal forces. The response of
the truss to the load $f$ is measured by nodal displacements collected
in a displacement vector $u\in\RR^{\tn}$. Some of the displacement
components may be restricted: a node can be fixed in a wall, then the
corresponding displacements are prescribed to be zero. The number of
free nodes multiplied by the spatial dimension will be denoted by $n$
and we will assume that $f\in\RR^n$ and $u\in\RR^n$.

We introduce the bar stiffness matrices $K_i$ and
assemble them in the global stiffness matrix of the truss
\begin{equation}
    K(t) = \sum_{i=1}^m t_i K_i  = \sum_{i=1}^m t_i \frac{E_i}{\ell_i^2}\gamma_i \gamma_i^\top ,\quad i=1,\ldots,m
\label{ama:SINGLE}
\end{equation}
and introduce the equilibrium equation
\begin{equation}
    K(t)u = f\,.
\label{equilibrium:SINGLE}
\end{equation}
\begin{Assumption}\label{assum:6}
	$K(\mathbf{1})$ with $\mathbf{1}=(1,1,\ldots,1)\in \mathbb{R}^m$ is positive definite and the load vector $f$ is in the range space of $K(\mathbf{1})$.
\end{Assumption}

\subsection{Truss topology problem, rank of the solution}

Let $0\leq\underline{t}_i \leq \overline{t}_i, \ i=1,\ldots,m$ and $\gamma$ be a positive constant. The basic truss topology problem reads as follows.
\begin{align}
    &\min_{t\in\RR^m,u\in\RR^n} \sum_{i=1}^m t_i\label{volume_b}\\
    &\mbox{subject to}\nonumber\\
    &\qquad K(t)u=f\nonumber\\
    &\qquad f^\top u \leq \gamma \nonumber\\
    &\qquad \underline{t}_i \leq t_i \leq \overline{t}_i, \quad i=1,\ldots,m\,.\nonumber
\end{align}

The following version of the Schur complement theorem can be found, e.g., in \cite{achtziger2008structural}.
\begin{lemma}\label{th:schur:SINGLE}
	Let $t\in\RR^m$, $t\geq 0$, and $\gamma\in\RR$ be fixed. Then there exists
	$u\in\RR^n$ satisfying
	$$
    	K(t) u = f \qquad \mbox{and} \qquad f^\top  u  \leq \gamma
	$$
	if and only if
	\[
    	\begin{pmatrix} \gamma &-f^\top \\ -f & K(t)\end{pmatrix}\succeq 0 \,.
	\]
\end{lemma}
From the above lemma, we immediately get an equivalent SDP formulation of the problem:
\begin{align}
    &\min_{t\in\RR^m} \sum_{i=1}^m t_i \label{v_sdp_b}\\
    &\mbox{subject to}\nonumber\\
    &\qquad   \begin{pmatrix}
        \gamma&-f^\top \\-f&K(t)
    \end{pmatrix} 
    \succcurlyeq 0 \nonumber\\
    &\qquad \underline{t}_i \leq t_i \leq \overline{t}_i, \quad i=1,\ldots,m \nonumber\,.
\end{align}

The dual to \cref{v_sdp_b} reads as follows.
\begin{align}
    &\max_{\stackrel{\scriptstyle X\in\S^{n+1}}{\overline{\rho}\in\RR^m,
    		\underline{\rho}\in\RR^m}} \begin{pmatrix}
    -\gamma&f^\top \\f&0
    \end{pmatrix}\bullet X - \sum_{i=1}^m \overline{\rho}_i \overline{t}_i
    + \sum_{i=1}^m \underline{\rho}_i \underline{t}_i \label{eq:sdp_dual:SINGLE}\\
    &\mbox{subject to}\nonumber\\
    &\qquad \begin{pmatrix}0&0\\0&K_i\end{pmatrix}\bullet X
    - \overline{\rho}_i +\underline{\rho}_i
    = 1, \quad i=1,\ldots,m \nonumber\\
    &\qquad X\succcurlyeq 0\nonumber\\
    &\qquad \overline{\rho}_i\geq 0,\ \underline{\rho}_i\geq 0,
    \quad i=1,\ldots,m \,.\nonumber
\end{align}

\begin{lemma}
  Problems \cref{v_sdp_b}, \cref{eq:sdp_dual:SINGLE} satisfy \cref{assum:a3} and \cref{assum:a4}.
\end{lemma}
\begin{proof}
    Every vector $\gamma_i$ in \cref{ama:SINGLE} has at most 4 non-zero elements (6 in 3D space) and thus every matrix $K_i$ has at most 16 (36) non-zero elements, independently of the size of the problem; hence \cref{assum:a3} is satisfied. \cref{assum:a4} is trivially satisfied for box constraints on~$t_i$.
\end{proof}

\begin{proposition}{\rm\cite[Chap.\;2.5]{achtziger_diss}}
	\label{th:epsilon} Let $t^*\in\RR^m$ be a solution of the minimum
	volume problem \cref{volume_b} with $\underline{t}_i=0$,
	$i=1,\ldots,m$. There exists a sequence $\{t_k\}_{k=1}^\infty$,
	$t_k\in\RR^m$, with the following properties
	\begin{itemize}
		\item $t_k$ is a solution of \cref{volume_b} with
		$\underline{t}_i=\varepsilon_k$, $i=1,\ldots,m$\, and
		$\varepsilon_k\to 0$ as $k\to\infty$\,;
		\item $t_k\to t^*$ as  $k\to\infty$\,.
	\end{itemize}
\end{proposition}

\begin{theorem}\label{th:raknone}
	There exists a solution $(X^*,\overline{\rho}^*,\underline{\rho}^*)
	\in \S^{n+1}\times\RR^m\times\RR^m$  of the dual SDP problem
	\cref{eq:sdp_dual:SINGLE} such that the rank of $X^*$ is one.
\end{theorem}
\begin{proof} Let $t^*$ be a solution of the primal SDP problem
	\cref{v_sdp_b} associated with \cref{eq:sdp_dual:SINGLE}. Denote by
	$S^*\in\S^{n+1}$ the primal slack matrix variable
	$$
    	S^* = \begin{pmatrix}
    	\gamma&-f^\top \\-f&K(t^*)
    	\end{pmatrix} \,.
	$$
	This matrix is complementary to any
	solution $X$ of the dual SDP problem, i.e.,
	$$
	    \langle X,S^* \rangle = 0 \,.
	$$
	As $X$ and $S^*$ are commuting and thus simultaneously
	diagonalizable, we have the following rank
	condition (see, e.g., \cite{deklerk}):
	\begin{equation}\label{eq:rank:SINGLE}
	    \rank(S^*) + \rank(X) \leq n+1 \,.
	\end{equation}
	
	If $\underline{t}_i>0$, $i=1,\ldots,m$, the matrix $K(t^*)$ is
	positive definite and, due to \Cref{assum:6}, has a full rank $n$.
	Hence, by the above rank condition, the rank of any dual solution
	$X$ is at most one. Since we exclude trivial solutions with $X\equiv
	0$, the rank of $X$ is exactly equal to one.
	
	Assume now that $\underline{t}_i=0$, $i=1,\ldots,m$, i.e., the
	matrix $K(t^*)$ can be rank deficient. Due to
	\cref{th:epsilon}, there is a sequence of solutions
	$\{S_k\}$ to the problem \cref{v_sdp_b} with
	$(\underline{t}_i)_k=\varepsilon_k$, $\varepsilon_k\to 0$ as $k\to
	\infty$, such that $S_k\to S^*$. Associated with $\{S_k\}$ there is
	a sequence of dual solutions $\{X_k\}$, such that any pair
	$(S_k,X_k)$ satisfies the complementarity condition and thus the
	rank condition. By this construction, the matrices $K(t_k)$ are
	strictly positive definite and thus, due to \Cref{assum:6}, have full
	rank $n$. From the rank condition \cref{eq:rank:SINGLE} it follows
	that the rank of $X_k$ is at most one and thus exactly equal to one.
	Hence there is a sequence of vectors $\{x_k\}$ such that $X_k=
	x_kx_k^\top $ for each $k$. Then $X^*=x^* (x^*)^\top $ is a
	solution of the dual SDP problem and thus complementary to $S^*$. By
	construction, $X^*$ is a rank-one matrix.
\end{proof}


\begin{corollary}\label{th:raknone_cor}
	Let $\underline{t}_i>0$ for $i=1,\ldots,m$. Then the solution $X^* \in \S^{n+1}$  of the dual SDP problem \cref{eq:sdp_dual:SINGLE} has rank one.
\end{corollary}

\paragraph{Problems with approximate low rank} The conclusion of \cref{th:raknone_cor} changes when we assume that $\underline{t}_i=0$ for $i=1,\ldots,m$. In this case, some bars may ``vanish" at the optimum (the corresponding $t_i=0$). Moreover, some nodes of the structure may only be connected by these vanishing bars and hence also ``vanish" from the optimal structure. As a consequence, the stiffness matrix $K(t)$ will become singular with as many zero eigenvalues as the vanished nodes (not connected to the optimal structure). Then
$$
  \rank(S^*)< n \quad\mbox{and thus}\quad \rank(X^*)>1\quad\mbox{is possible}\,.
$$
Hence, while there is always a rank-one solution by \cref{th:raknone}, in this case there might be other solutions with higher rank. In particular, when we solve the problems by an interior point method, the pair of solutions $(S^*,X^*)$ will be maximally complementary (\cite[Thm.\,3.4]{deklerk}), so the rank of $X^*$ will be equal to the number of the vanished nodes plus one. 

\begin{Example}
\Cref{fig:tr1} shows an example of a truss with potentially six nodes (the remaining three nodes are fixed by boundary conditions); at the optimum, nodes 1, 2, 3, 4, 6 vanished---they are not connected to the optimal structure but are still present in the problem and its stiffness matrix. Hence the solution of the dual problem may have rank of up to $(1+2\times 5)=11$.
	\begin{figure}[tbhp]
	\begin{center}
		\resizebox{0.5\hsize}{!}
		{\includegraphics{{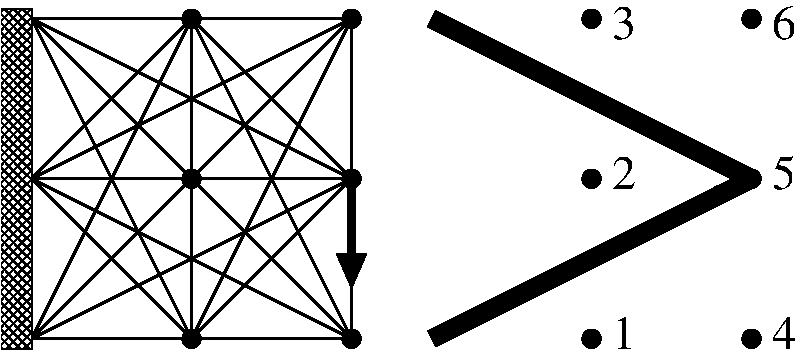}}}
	\end{center}
	\caption{A problem with five nodes vanishing at the optimum; initial structure (left) and optimal solution (right).}
	\label{fig:tr1}
	\end{figure}
The associated \cref{tab:ex1} presents, column-wise, eigenvalues of $X^*$ and $S^*$ for $\underline{t}>0$ (when rank-one solution $X^*$ is expected) and eigenvalues of $X^*$ and $S^*$ for $\underline{t}=0$ (when one outlying eigenvalue of $X^*$ is expected).

\noindent
\begin{table}[tbhp]
\begin{center}
\caption{Eigenvalues of $X^*$ and $S^*$ for problem in \cref{fig:tr1}.}
\label{tab:ex1}
\begin{minipage}[t]{0.22\hsize}
{\scriptsize
\centering{$\lambda(X^*),\ \underline{t}>0$}
\begin{verbatim}
1.980482521248879
0.000000102260913
0.000000086085241
0.000000068216515
0.000000054110869
0.000000042818584
0.000000039334355
0.000000033270286
0.000000032811718
0.000000028005339
0.000000027431494
0.000000000001029
0.000000000000041
\end{verbatim}}
\end{minipage}\quad
\begin{minipage}[t]{0.22\hsize}
{\scriptsize
\centering{$\lambda(S^*),\ \underline{t}>0$}
\begin{verbatim}
0.000000000000102
0.000003116874928
0.000003956920433
0.000006409263345
0.000007529457678
0.000009477992158
0.000011708852890
0.000012193659241
0.000014115281046
0.000014761068984
0.000016565957597
0.399999306252456
10.10000007997102
\end{verbatim}}
\end{minipage}\quad
\begin{minipage}[t]{0.22\hsize}
{\scriptsize
\centering{$\lambda(X^*),\ \underline{t}=0$}
\begin{verbatim}
2.064483254260574
0.010238711360358
0.009214785863400
0.005367091514209
0.005013654697043
0.003545367975067
0.002923233469802
0.002610639717586
0.001846306250799
0.001771660792931
0.001555022631457
0.000000000000031
0.000000000000001
\end{verbatim}}
\end{minipage}\quad
\begin{minipage}[t]{0.22\hsize}
{\scriptsize
\centering{$\lambda(S^*),\ \underline{t}=0$}
\begin{verbatim}
0.000000000010001
0.000000000013521
0.000000000013687
0.000000000016235
0.000000000017372
0.000000000018683
0.000000000020485
0.000000000021874
0.000000000024806
0.000000000028448
0.000000000032822
0.400000000006957
10.10000000001012
\end{verbatim}}
\end{minipage}
\end{center}
\end{table}
\end{Example}

\medskip
Notice that the number of the vanished nodes depends on the initial setting of the problem and can be rather high, as compared to the total number of the nodes. 

\subsection{Truss topology problem, vibration constraints}
Formulation of the basic truss topology problem \cref{volume_b} as an SDP is rather academic. It was shown, e.g., in \cite{jarre1998optimal} that a dual to \cref{volume_b} is a convex nonlinear programming problem that can be solved very efficiently by interior point methods. The SDP formulation gains significance once we add more, important and practical, constraints to the problem. In particular, it was shown in \cite{achtziger2007maximization,achtziger2008structural} that a constraint on natural (free) vibrations of the optimal structure leads to a linear matrix inequality and that this is, arguably, the best way how to treat this constraint.

The free vibrations are the squares of the eigenvalues of the following
generalized eigenvalue problem
\begin{equation}\label{eq:vibEVP}
    K(t) w = \lambda \left(M(t)+M_0\right) w \,.
\end{equation}
Here
$M(t)=\sum\limits_{i=1}^{m} t_i M_i$
is the so-called mass matrix that collects information about the mass
distribution in the truss. The matrices $M_i$ are positive semidefinite
and have the same sparsity structure as $K_i$. The non-structural mass
matrix $M_0$ is a constant, typically diagonal matrix with very few
nonzero elements.

Low vibrations are dangerous and may lead to structural collapse. Hence
we typically require the smallest free vibration to be bigger than some
threshold, that is:
$$
    \lambda_{\min} \geq \overline{\lambda}\quad\mbox{for a given}\ \overline{\lambda}>0\,\
$$
where $\lambda_{\min}$ is the smallest eigenvalue of the generalized eigenvalue problem \cref{eq:vibEVP}.
This constraint can be  equivalently written as a linear matrix
inequality
\begin{equation} \label{eq:truss_vibra}
    K(t)-\overline{\lambda}\left(M(t)+M_0\right)\succcurlyeq 0
\end{equation}
which is to be added to the basic truss topology problem.
As \cref{eq:truss_vibra} is a linear matrix inequality in variable
$t$, it is natural to add this constraint to the primal SDP formulation
\cref{v_sdp_b}. We will thus get the following linear SDP formulation
of the truss topology design with a vibration constraint:
\begin{align}
    &\min_{t\in\RR^m} \sum_{i=1}^m t_i \label{eq:vib}\\
    &\mbox{subject to}\nonumber\\
    &\qquad   \begin{pmatrix}
    \gamma&-f^\top \\-f&K(t)
    \end{pmatrix} \succcurlyeq 0 \nonumber\\
    &\qquad K(t)-\overline{\lambda}\left(M(t)+M_0\right)\succcurlyeq 0\nonumber\\
    &\qquad \underline{t}_i \leq t_i \leq \overline{t}_i, \quad i=1,\ldots,m \nonumber\,.
\end{align}

It has been shown, e.g., in \cite{stingl2009sequential} that, when $\underline{t}_i>0$, the rank of the optimal dual variable to the second matrix inequality is equal to the multiplicity of the smallest eigenvalue of \cref{eq:vibEVP}. This multiplicity depends on the geometry of the optimal structure but is, typically very low, usually not bigger than 1 or 2. Thus the new problem \cref{eq:truss_vibra} fits in our framework of SDPs with very low rank dual solutions.
\section{Numerical experiments}\label{sec:examples}
\subsection{Problem database}
To test our algorithms, we generated a library of truss topology problems of various sizes, both with and without the vibration constraint. All problems have the same geometry, boundary conditions and loads and only differ in the number of potential nodes and bars. 

The geometry and loading for TTO without vibration (formulation~\cref{v_sdp_b}) is as in \cref{fig:tr1}-left, while the data for problems with the vibration constraint (formulation~\cref{eq:vib}) differ only in the horizontal orientation of the load vector. We have generated four groups of problems with the following naming convention. In each group, all nodes in the initial ground structure are connected by potential bars.
\begin{verse}
    \item[\bf tru<n>] standard TTO problem~\cref{v_sdp_b} with $\underline{t}=0$, discretized by $n\times n$ nodes;
    \item[\bf tru<n>e] standard TTO problem~\cref{v_sdp_b} with $\underline{t}=\varepsilon>0$, discretized by $n\times n$ nodes;
    \item[\bf vib<n>]  TTO problem with the vibration constraint \cref{eq:vib} with $\underline{t}=0$, discretized by $n\times n$ nodes;
    \item[\bf vib<n>e]  TTO problem with the vibration constraint~\cref{eq:vib} with $\underline{t}=\varepsilon>0$, discretized by $n\times n$ nodes.
\end{verse}
\Cref{tab:sizes_tru} present dimensions of the generated problems. Notice that these problems satisfy \cref{assum:a5}. 
\begin{table}[tbhp]
{\footnotesize\caption{Problems tru<n> and tru<n>e and problems vib<n> and vib<n>e, number of variables $n$, size of the LMI constraint $m$ and number of linear constraints.}
    \label{tab:sizes_tru}
    \begin{minipage}[t]{0.48\hsize}
    \centering
    \begin{tabular}{cccc}\toprule
	problem & $n$ & $m$ & lin.\ constr.\\\midrule
tru3(e)	&36	    &13&72\\
tru5(e)	&300	&41&600\\
tru7(e)	&1176	&85&2352\\
tru9(e)	&3240	&145&6480\\
tru11(e)	&7260	&221&14520\\
tru13(e)	&14196	&313&28392\\
tru15(e)	&25200	&421&50400\\
tru17(e)	&41616	&545&83232\\
tru19(e)	&64980	&685&129960\\
tru21(e)	&97020	&841&194040\\
tru23(e)	&139656	&1013&279312\\
tru25(e)	&195000	&1201&390000\\\bottomrule
    \end{tabular}
    \end{minipage}
    \begin{minipage}[t]{0.48\hsize}
    \centering
    \begin{tabular}{cccc}\toprule
	problem & $n$ & $m$ & lin.\ constr.\\\midrule
vib3(e)	    &36	    &(13,\,12)&72\\
vib5(e)	    &300	&(41,\,40)&600\\
vib7(e)	    &1176	&(85,\,84)&2352\\
vib9(e)	    &3240	&(145,\,144)&6480\\
vib11(e)	&7260	&(221,\,220)&14520\\
vib13(e)	&14196	&(313,\,312)&28392\\
vib15(e)	&25200	&(421,\,420)&50400\\
vib17(e)	&41616	&(545,\,544)&83232\\
vib19(e)	&64980	&(685,\,684)&129960\\
vib21(e)	&97020	&(841,\,840)&194040\\
vib23(e)	&139656	&(1013,\,1012)&279312\\
vib25(e)	&195000	&(1201,\,1200)&390000\\\bottomrule
    \end{tabular}
    \end{minipage}
    }
\end{table}

For illustration, \cref{fig:tr_opt} shows optimal results for problems tru25 and vib19.
	\begin{figure}[tbhp]
	\begin{center}
		\resizebox{0.35\hsize}{!}
		{\includegraphics{{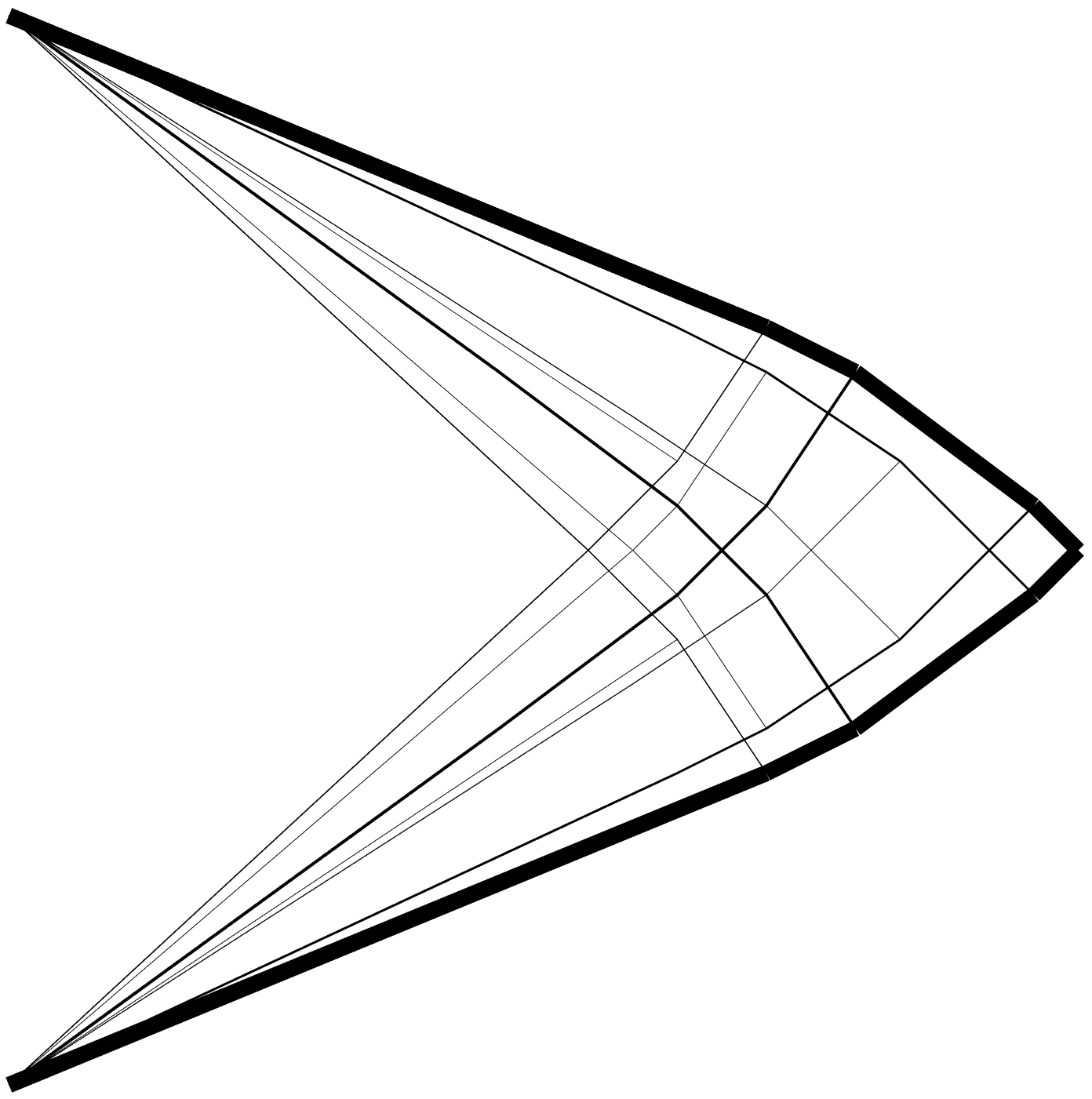}}}\qquad
				\resizebox{0.35\hsize}{!}
		{\includegraphics{{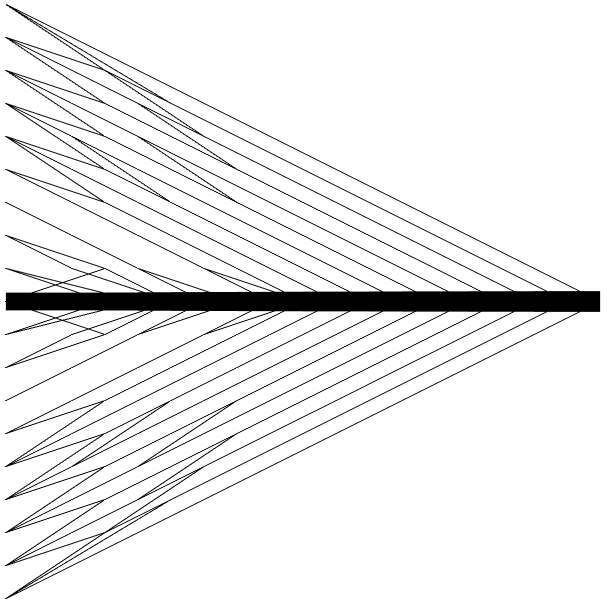}}}
	\end{center}
	\caption{Optimal results for problems tru25 (left) and vib19 (right).}
	\label{fig:tr_opt}
	\end{figure}

\subsection{Numerical results}
In this section we will present numerical results for the truss topology problems from the problem database. The two algorithms were implemented in MATLAB codes PDAL and Loraine (Low-rank interior-point). The codes use the same modules for data input and for the solution of the linear systems. 
Both codes use DIMACS stopping criteria \cite{dimacs} with $\varepsilon_{\scriptscriptstyle \rm DIMACS}=10^{-5}$.

In \Cref{tab:PDAL} we present the value of parameters chosen in PDAL (notice that they differ for the tru and vib collections). Moreover for the block constraints in the tru problems we are using the Penalty-Barrier function
with $\tau=0.5$ while for the vib problems we use a pure barrier function. Loraine parameters have already been specified in \cref{algo:IP}.

The stopping parameter for the CG method $\varepsilon_{\scriptscriptstyle \rm CG}$ in \cref{algo:PCG} has been set and updated as follows: we start with $\varepsilon_{\scriptscriptstyle \rm CG} = 0.01$ and multiply it after every major iteration by 0.5, until it reaches the value $10^{-6}$; then we continue with this value till convergence of the optimization algorithm.

\begin{table}[htbp]
    \centering
\label{tab:PDAL}%
  {\footnotesize\caption{Parameters used in PDAL for the tru and vib problems}
    \begin{tabular}{ccc|ccc}
    \toprule
 param & tru & vib & param & tru & vib \\  
 \midrule
 $\pi^{\rm{lin}}_{\rm{min}}$ & $10^{-9}$ & $10^{-11}$&  
  $\gamma^{\rm{lin}}$ & 0.5 & 1 \\[1ex]
 $\pi^{\rm{lmi}}_{\rm{min}}$ & $10^{-5}$ & $10^{-5}$ &  
 $\gamma^{\rm{lmi}}$ & 0.5 & 0.8 \\ [1ex]
 $\pi^{\rm{lin}}_{\rm{upd}}$ & 0.5 & 0.3&
 $r$ & 0.01 & 0.01  \\ [1ex]
  $\pi^{\rm{lmi}}_{\rm{upd}}$ & 0.5 & 0.3&
 $\varepsilon$ & $10^{-6}$ &  $10^{-6}$\\[1ex]
    \bottomrule
    \end{tabular}%
  }
\end{table}%
%

Unless stated otherwise, we use the following preconditioners:
\begin{description}
    \item[PDAL:] The preconditioner $H_\gamma$ described in \Cref{sec:prec_AL} is used throughout all iterations.
    \item[Loraine:] A hybrid preconditioner called $H_{\rm hyb}$: we start with $H_\beta$ \cref{eq:Hbeta} until the criterion below is satisfied and then switch to $H_\alpha$ \cref{Halpha}:
    $$
      N_{\rm cg\_iter}>kp\sqrt{n}/10\ \  \&\ \  N_{\rm iter}>\sqrt{n}/60\,.
    $$
    Here $N_{\rm iter}$ is the current IP iteration and $N_{\rm cg\_iter}$ the number of CG steps needed to solve the corrector equation in this iteration. This criterion is, obviously, purely heuristic. We further choose the value of $\tau$ in $H_\alpha$ and $H_\beta$ as 
    $$
    \tau = \lambda_1(W_i)+0.5\,{\rm mean}(\lambda_1(W_i),\ldots,\lambda_{m-k_i}(W_i))\,.
    $$
\end{description}

All problems were solved on an iMac with 3.6 GHz 8-Core Intel Core i9 and 40 GB 2667 MHz DDR4 using MATLAB R2020a.

\smallskip
Before presenting the results, we want to illustrate the efficiency of the used preconditioners, in particular, in the light of the assumption of \cref{th:precond1}. 
	\begin{figure}[tbhp]
	\begin{center}
		\resizebox{0.45\hsize}{!}
		{\includegraphics{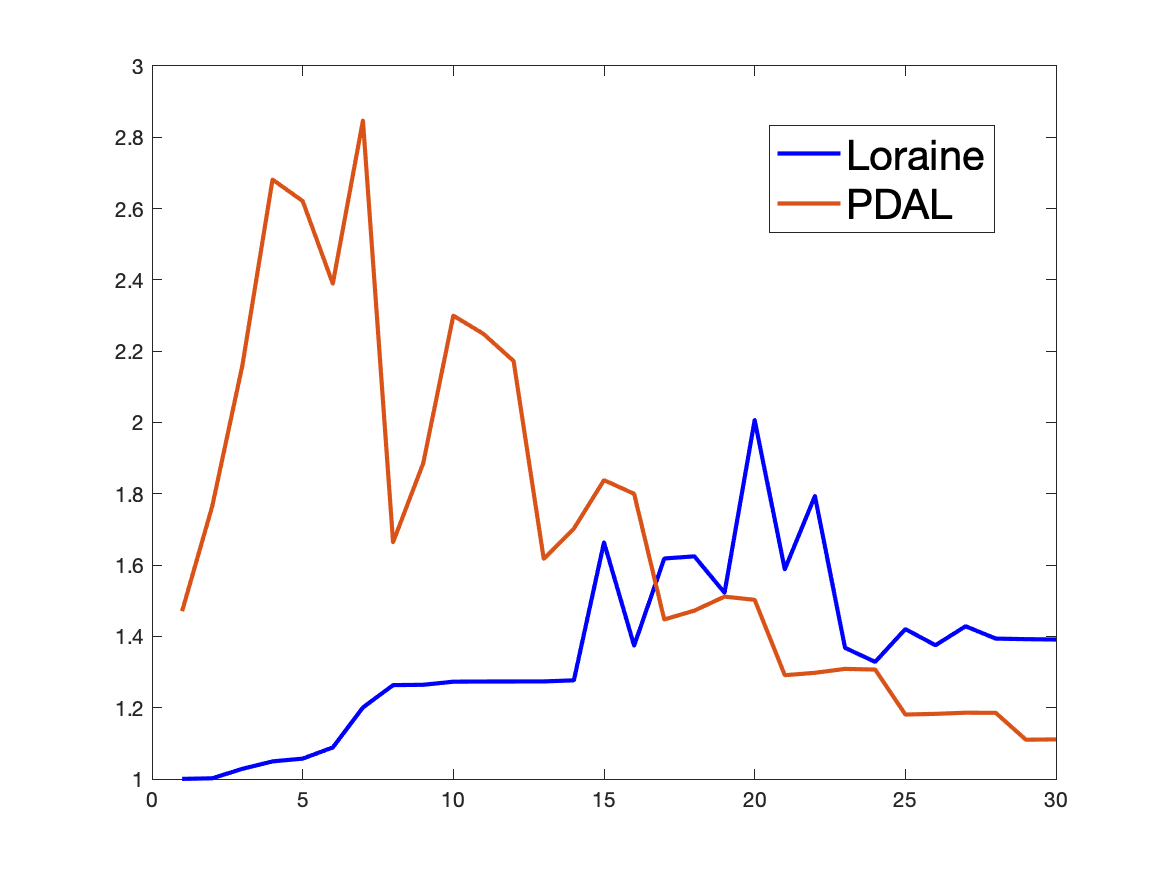}}
				\resizebox{0.45\hsize}{!}
		{\includegraphics{{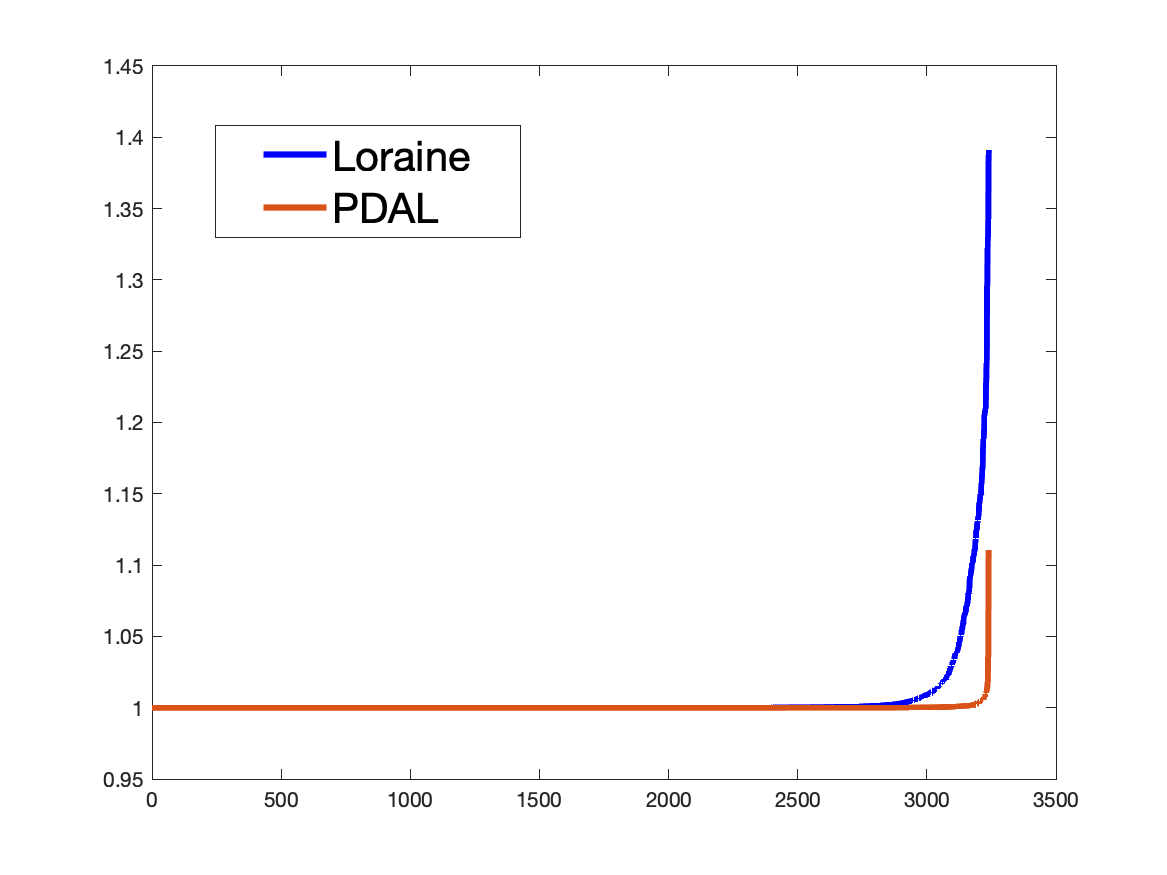}}}
	\end{center}
	\caption{Left: The history of the last 30 condition numbers of the preconditioned matrix in Loraine (blue) and PDAL (red) in tru9 problem. Right: Eigenvalues of the preconditioned matrix in the last iteration of Loraine (blue) and PDAL (red).}
	\label{fig:tr_IP_PDAL_cond}
	\end{figure}
In \cref{fig:tr_IP_PDAL_cond}--left we present the history of the condition numbers of the preconditioned matrices for problem tru9; a complete history for Loraine and the last 30 iterations for PDAL. While the history differs for the two algorithms, the condition numbers are very low and, in case of PDAL, decrease when converging to the solution. For this problem, we also computed the values of $\underline{\varepsilon}$ and $\overline{\varepsilon}$ from \cref{th:precond1} and found that the estimate of the condition number in this theorem was always sharp. \cref{fig:tr_IP_PDAL_cond}--right depicts all eigenvalues of the preconditioned matrix in the last iteration. Most of the eigenvalues are very close or equal to one and the matrices have only a few outlying eigenvalues, albeit still not far from one.

We now present results for all tru and vib problems using Loraine and PDAL. These are shown in \cref{tab:tru,tab:vib}. The tables present the number of iterations of the respective algorithm, the total number of CG iterations using the hybrid preconditioner $H_{\rm hyb}$ in case of Loraine and $H_\alpha$ in case of PDAL. This is followed by CPU time spent in the optimization solver and time per iteration. Notice that PDAL solves one linear system per iteration, while Loraine solves two (predictor and corrector).
For all problems we set the expected rank of the dual solutions $k=1$.
\begin{table}[htbp]
  \label{tab:tru}%
  \centering
  {\footnotesize\caption{Loraine and PDAL in tru$*$ problems}
    \begin{tabular}{l|rrrr|rrrr}        \toprule                                   
    \multicolumn{1}{r}{} & \multicolumn{4}{|c|}{Loraine} & \multicolumn{4}{c}{PDAL}   \\
    \multicolumn{1}{r}{} & \multicolumn{1}{|r}{iter} & \multicolumn{1}{r}{CG iter} & \multicolumn{1}{r}{CPU} & \multicolumn{1}{r}{CPU/iter} & \multicolumn{1}{|r}{iter} & \multicolumn{1}{r}{CG iter} & \multicolumn{1}{r}{CPU} & \multicolumn{1}{r}{CPU/iter}   \\\midrule
    tru3  & 16    & 122   & 0.01  & 0.00  & 35    & 107   & 0.04  & 0.00    \\
    tru5  & 21    & 190   & 0.08  & 0.00  & 61    & 357   & 0.28  & 0.00    \\
    tru7  & 27    & 236   & 0.24  & 0.01  & 62    & 310   & 0.54  & 0.01    \\
    tru9  & 31    & 333   & 0.65  & 0.02  & 69    & 374   & 1.5  & 0.02    \\
    tru11 & 36    & 370   & 1.6  & 0.04  & 78    & 435   & 3.4   & 0.04    \\
    tru13 & 45    & 500   & 4.5   & 0.10  & 102   & 818   & 10    & 0.10    \\
    tru15 & 52    & 882   & 11  & 0.22  & 97    & 670   & 20    & 0.21    \\
    tru17 & 53    & 980   & 24    & 0.45  & 96    & 744   & 39    & 0.41    \\
    tru19 & 64    & 1310  & 51    & 0.80  & 101   & 808   & 71    & 0.70    \\
    tru21 & 65    & 1325  & 84    & 1.29  & 111   & 965   & 151   & 1.36    \\
    tru23 & 72    & 2450  & 189   & 2.63  & 112   & 1019  & 270   & 2.41    \\
    tru25 & 87    & 2148  & 317   & 3.64  & 142   & 1929  & 634   & 4.46    \\
    \bottomrule
    \end{tabular}%
  }
\end{table}%
\begin{table}[htbp]
  \label{tab:vib}%
  \centering
  {\footnotesize\caption{Loraine and PDAL in vib$*$ problems}
    \begin{tabular}{l|rrrr|rrrr}        \toprule                                   
    \multicolumn{1}{r}{} & \multicolumn{4}{|c|}{Loraine} & \multicolumn{4}{c}{PDAL}   \\
    \multicolumn{1}{r}{} & \multicolumn{1}{|r}{iter} & \multicolumn{1}{r}{CG iter} & \multicolumn{1}{r}{CPU} & \multicolumn{1}{r}{CPU/iter} & \multicolumn{1}{|r}{iter} & \multicolumn{1}{r}{CG iter} & \multicolumn{1}{r}{CPU} & \multicolumn{1}{r}{CPU/iter}   \\\midrule
    vib3  & 20    & 209   & 0.04  & 0.00  & 66    & 251   & 0.2   & 0.00  \\
    vib5  & 31    & 411   & 0.21  & 0.01  & 69    & 434   & 0.45  & 0.01  \\
    vib7  & 39    & 501   & 0.67  & 0.02  & 69    & 372   & 0.91  & 0.01  \\
    vib9  & 47    & 663   & 2.2  & 0.05  & 82    & 473   & 3.0     & 0.04  \\
    vib11 & 59    & 995   & 6.1  & 0.10  & 99    & 464   & 7.6   & 0.08  \\
    vib13 & 69    & 1153  & 16    & 0.23  & 112   & 508   & 18    & 0.16  \\
    vib15 & 80    & 1480  & 37    & 0.46  & 131   & 534   & 44    & 0.34  \\
    vib17 & 94    & 1781  & 92    & 0.98  & 130   & 501   & 91    & 0.70  \\
    vib19 & 108   & 2116  & 189   & 1.75  & 145   & 536   & 177   & 1.22  \\
    vib21 & 120   & 2844  & 401   & 3.34  & 173   & 723   & 456   & 2.64  \\
    vib23 & 130   & 2720  & 718   & 5.52  & 193   & 799   & 961   & 4.98  \\
    vib25 & 143   & 3752  & 1321  & 9.24  & 205   & 808   & 1724  & 8.41  \\
    \bottomrule
    \end{tabular}%
  }
\end{table}%

The following \cref{fig:tr_IP_PDAL,fig:vib_IP_PDAL,fig:tr_IP_PDAL_iter,fig:vib_IP_PDAL_iter} depict graphs of CPU times versus the number of variables for all four groups of problems, presented in the log-log scale. We can see that, while PDAL needs more iterations, the times per iteration are almost identical for the two codes. In particular, the time grows almost linearly (factor $\sim$1.25) with the size of the problem.
	\begin{figure}[tbhp]
	\begin{center}
		\resizebox{0.45\hsize}{!}
		{\includegraphics{{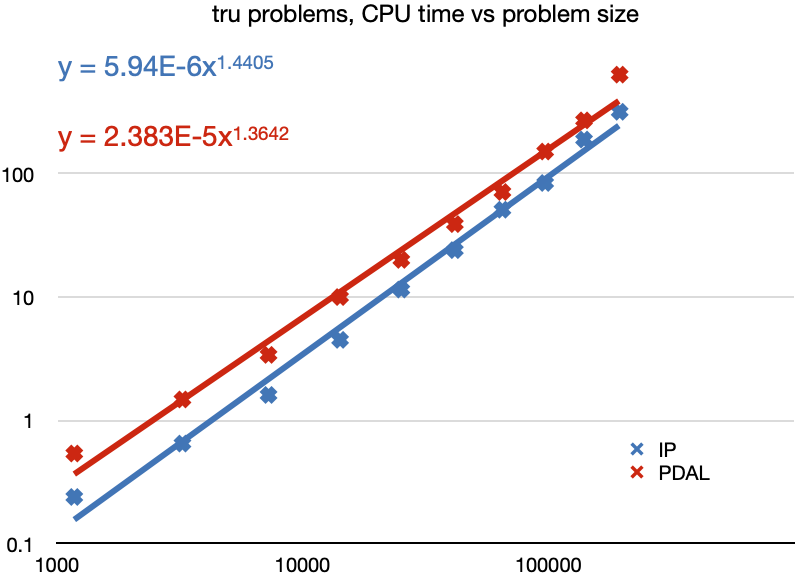}}}
				\resizebox{0.45\hsize}{!}
		{\includegraphics{{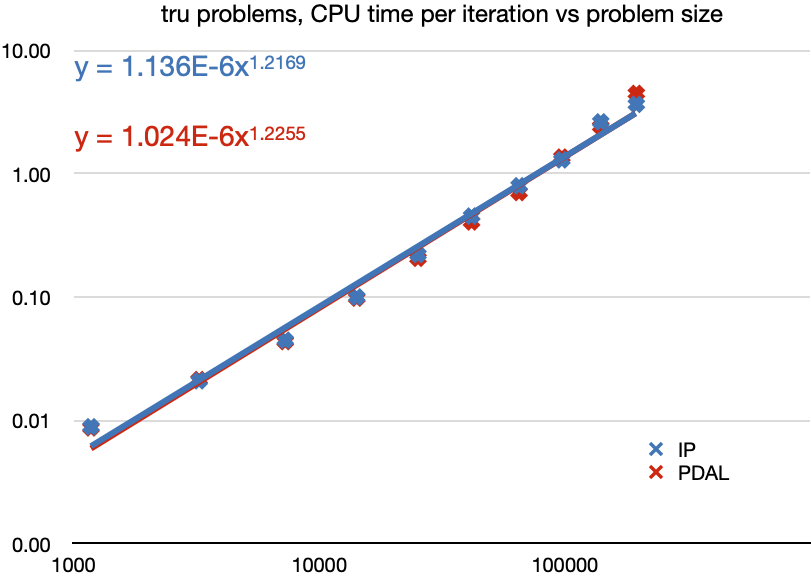}}}
	\end{center}
	\caption{Loraine (blue) and PDAL (red) in tru$*$ problems. CPU time vs number of variables in log-log scale.}
	\label{fig:tr_IP_PDAL}
	\end{figure}
	\begin{figure}[tbhp]
	\begin{center}
		\resizebox{0.45\hsize}{!}
		{\includegraphics{{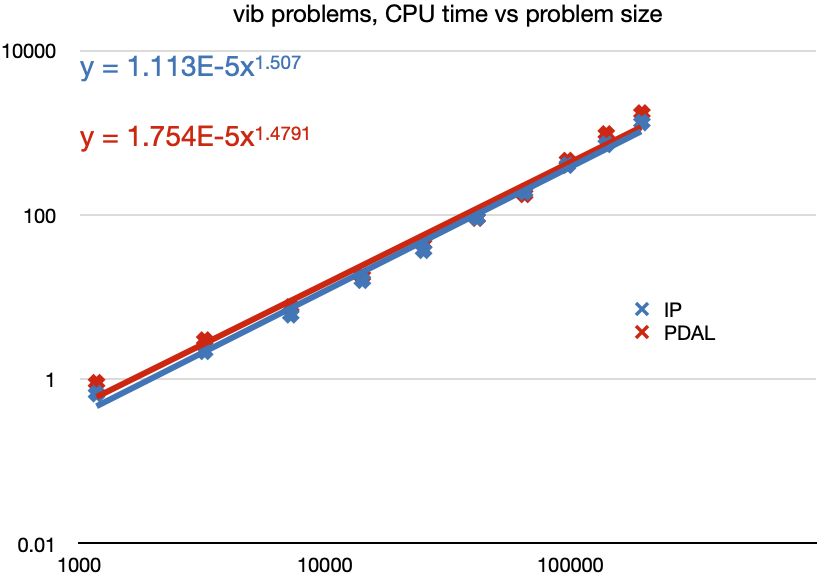}}}
				\resizebox{0.45\hsize}{!}
		{\includegraphics{{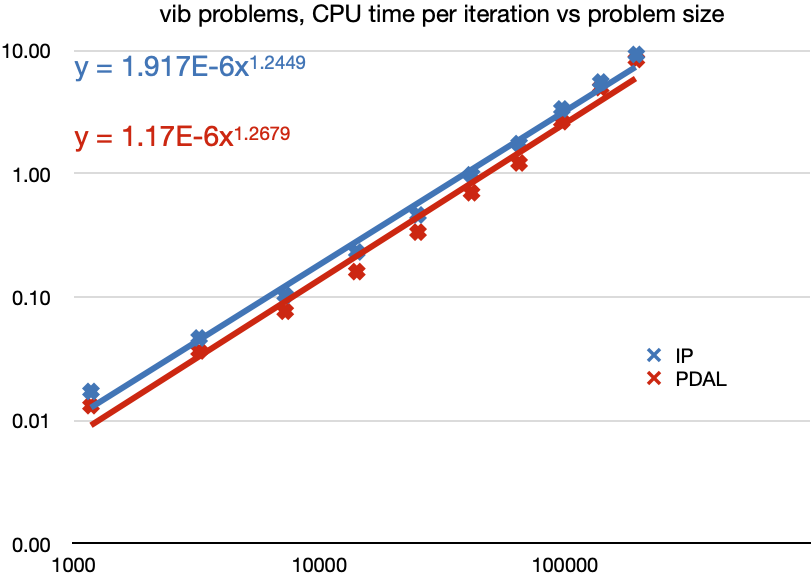}}}
	\end{center}
	\caption{Loraine (blue) and PDAL (red) in vib$*$ problems. CPU time vs number of variables in log-log scale.}
	\label{fig:vib_IP_PDAL}
	\end{figure}
	\begin{figure}[tbhp]
	\begin{center}
		\resizebox{0.45\hsize}{!}
		{\includegraphics{{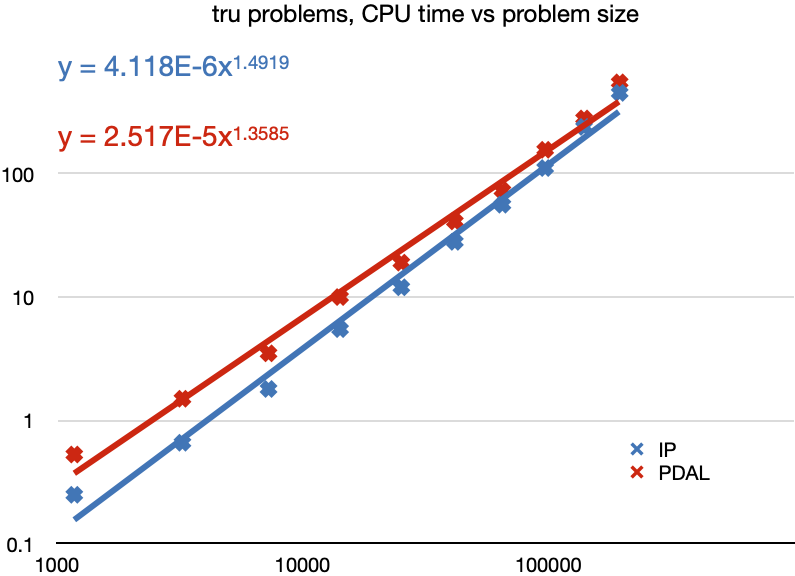}}}
				\resizebox{0.45\hsize}{!}
		{\includegraphics{{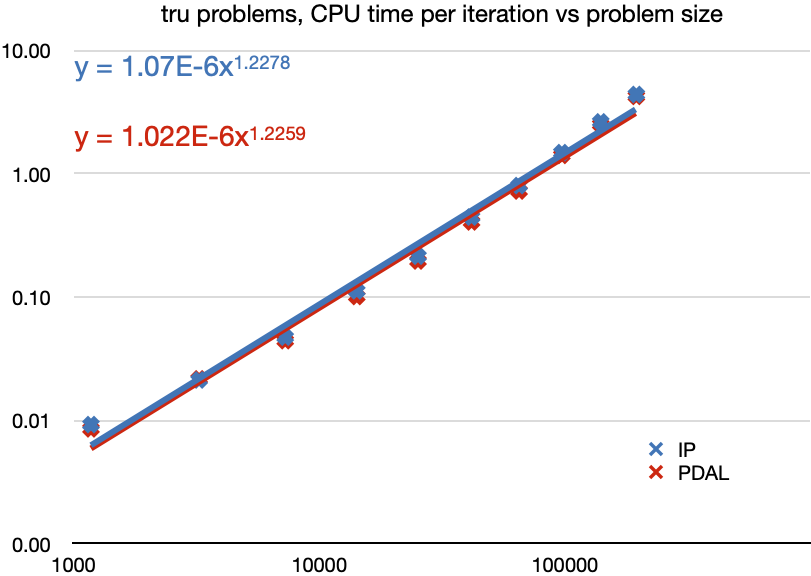}}}
	\end{center}
	\caption{Loraine (blue) and PDAL (red) in tru$*$e problems. CPU time vs number of variables in log-log scale.}
	\label{fig:tr_IP_PDAL_iter}
	\end{figure}
	\begin{figure}[tbhp]
	\begin{center}
		\resizebox{0.45\hsize}{!}
		{\includegraphics{{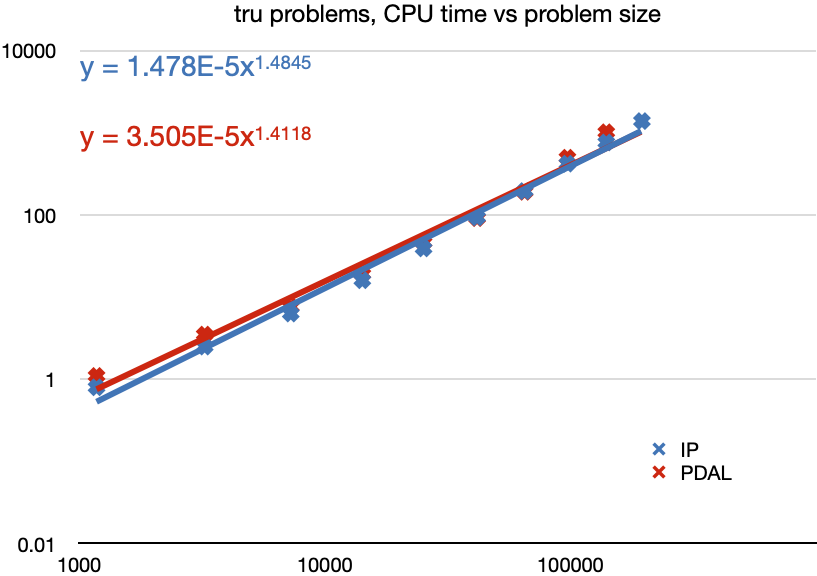}}}
				\resizebox{0.45\hsize}{!}
		{\includegraphics{{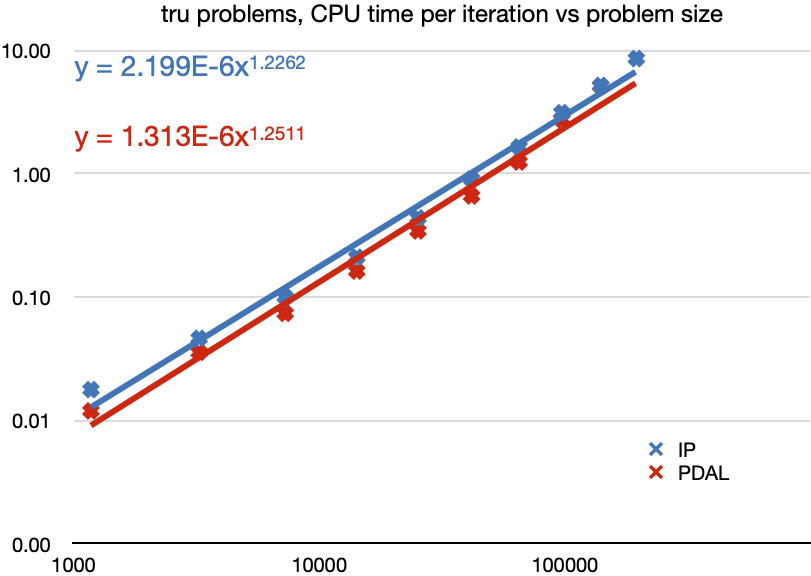}}}
	\end{center}
	\caption{Loraine (blue) and PDAL (red) in vib$*$e problems. CPU time vs number of variables in log-log scale.}
	\label{fig:vib_IP_PDAL_iter}
	\end{figure}


\subsection{Other software}
In this section we give a brief comparison Loraine and PDAL with other SDP software. First we present a comparison of Loraine, our IP implementation, with the solver MOSEK \cite{mosek} that is also based on an interior-point algorithm. While this comparison may seem unfair, as MOSEK uses a direct solver to solve the linear systems, we believe that it gives a good perspective of the clear advantage of iterative solvers with good preconditioners, when available and when applicable.
The results are presented in \cref{tab:mosek}; we again give the number of iterations of the solver, CPU time and CPU time per one iteration. The table also shows the speed-up with respect to MOSEK and speed-up per iteration. We only present results for problems tru3--tru17; for the larger problems MOSEK exceeded the available 40GB RAM.
%
\begin{table}[htbp]
  \centering
  {\footnotesize\caption{Loraine and MOSEK in tru$*$ problems}
    \begin{tabular}{l|rrr|rrrrrr}
    \toprule
    \multicolumn{1}{r}{} & \multicolumn{3}{|c|}{Mosek} & \multicolumn{6}{c}{Loraine} \\
   \multicolumn{1}{r}{problem} & \multicolumn{1}{|r}{iter} & \multicolumn{1}{r}{CPU} & \multicolumn{1}{r|}{CPU/iter} & \multicolumn{1}{r}{iter} & \multicolumn{1}{r}{CG iter} & \multicolumn{1}{r}{CPU} & \multicolumn{1}{r}{CPU/iter} & \multicolumn{1}{r}{speed-up} & \multicolumn{1}{r}{s-up/iter} \\\midrule
    tru3  & 14    & 0.22  & 0.02  & 16    & 122   & 0.01  & 0.00  & 22 & 25 \\
    tru5  & 15    & 0.23  & 0.02  & 21    & 190   & 0.08  & 0.00  & 2.9  & 4.0 \\
    tru7  & 17    & 0.78  & 0.05  & 27    & 236   & 0.24  & 0.01  & 3.3  & 5.2 \\
    tru9  & 20    & 5.6   & 0.28  & 31    & 333   & 0.65  & 0.02  & 8.6  & 13 \\
    tru11 & 26    & 44    & 1.69  & 36    & 370   & 1.61  & 0.04  & 27 & 37 \\
    tru13 & 29    & 235   & 8.10  & 45    & 500   & 4.5   & 0.10  & 52 & 81 \\
    tru15 & 32    & 1079  & 33.72 & 52    & 882   & 11.5  & 0.22  & 93 & 152 \\
    tru17 & 37    & 4671  & 126.24 & 53    & 980   & 24    & 0.45  & 194 & 278 \\
    \bottomrule
    \end{tabular}%
  \label{tab:mosek}%
  }
\end{table}%

In \cref{fig:mosek} we give a comparison of the CPU time per iteration. While MOSEK actually exceeds the theoretical complexity with coefficient 2.2, the complexity of Loraine is almost linear for these problems.
	\begin{figure}[tbhp]
	\begin{center}
		\resizebox{0.5\hsize}{!}
		{\includegraphics{{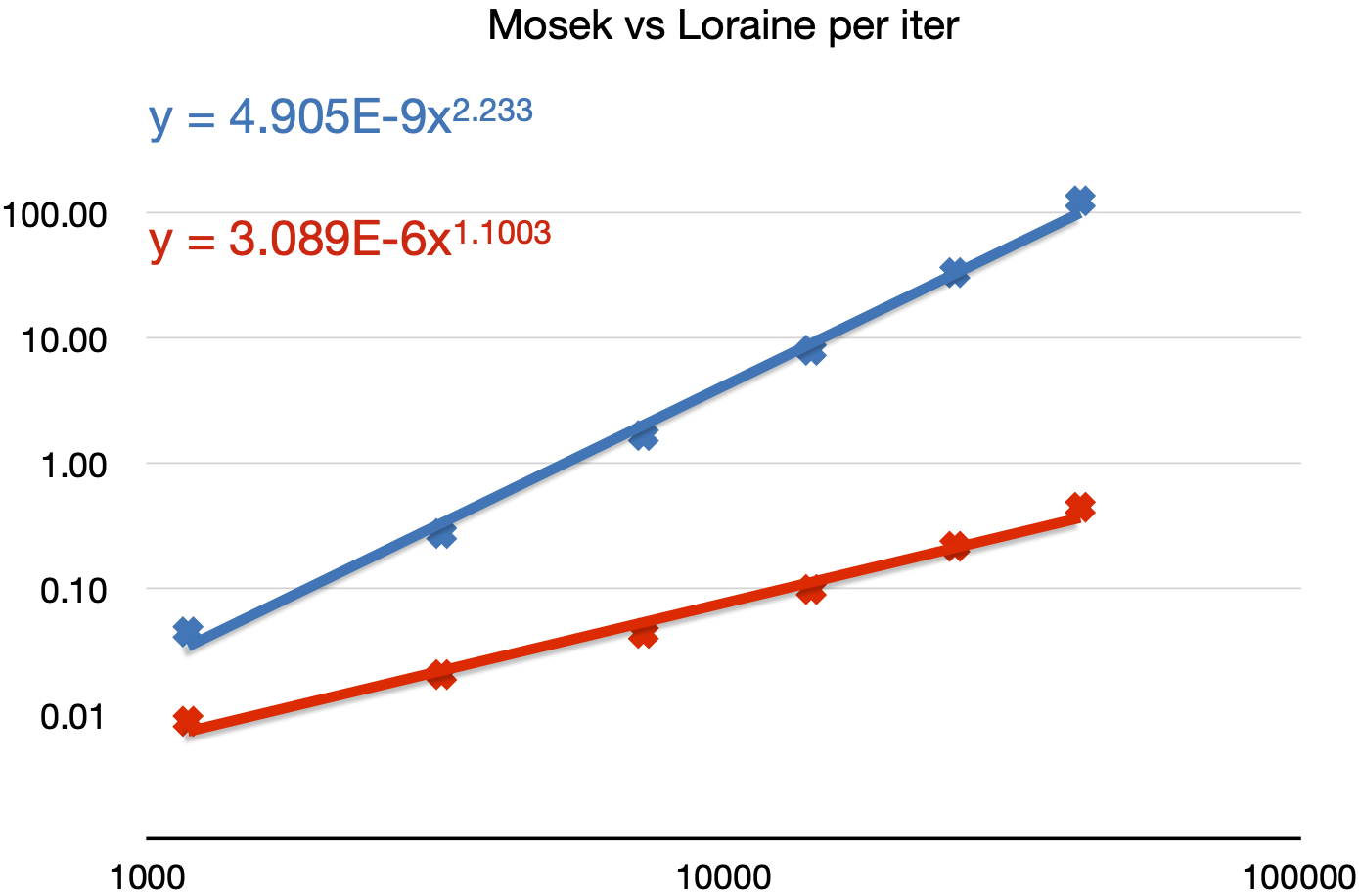}}}
	\end{center}
	\caption{Loraine (red) and MOSEK (blue) in tru$*$ problems. CPU time per iteration vs number of variables in log-log scale.}
	\label{fig:mosek}
	\end{figure}

The next \cref{tab:other} gives results of Loraine with the preconditioner ${H_\alpha}$ compared with other solvers that could potentially benefit from the low rank information: Loraine with the preconditioner $\widetilde{H}$ by Zhang and Lavaei \cite{Zhang_2017}; SDPNAL+ \cite{sdpnal}, an implementation of semismooth Newton-CG augmented Lagrangian method; and SDPLR \cite{burer-monteiro} using a nonlinear programming algorithm based on low-rank factorization of the variables. SDPNAL+ and SDPLR were used with default stopping tolerance. None of these codes can solves problems bigger than tru11 before reaching the maximum number of iterations, either internal or global. While Loraine solutions satisfy DIMACS criteria with $10^{-5}$, other codes often finish with lower precision, as documented by the final value of the objective function.
\begin{table}[htbp]
  \centering
  {\footnotesize\caption{Other solvers}
    \begin{tabular}{l|rrr|rrr|rr|rr}\toprule
    \multicolumn{1}{r}{} & \multicolumn{3}{|c}{{Loraine with ${H_\alpha}$}}  & \multicolumn{3}{|c}{{Loraine with $\widetilde{H}$}}
    & \multicolumn{2}{|c}{{SDPNAL+}} & \multicolumn{2}{|c}{{SDPLR}} \\
    \multicolumn{1}{r}{} & \multicolumn{1}{|r}{{iter}} & \multicolumn{1}{r}{{CG iter}} & \multicolumn{1}{r}{{CPU}} & 
    \multicolumn{1}{|r}{{iter}} & \multicolumn{1}{r}{{CG iter}} & \multicolumn{1}{r}{{CPU}} &  \multicolumn{1}{|r}{{iter}} & \multicolumn{1}{r}{{CPU}} &  \multicolumn{1}{|r}{{iter}} & \multicolumn{1}{r}{{CPU}}  \\\midrule
    {tru3}  & 16 & 122 & 0.01  & 16    & 131   & 0.02   & 440   & 0.5    & 15    & 0.01   \\
    {tru5}  & 21 & 190 & 0.08  & 21    & 197   & 0.3    & 9674  & 29     & 18    & 1.7    \\
    {tru7}  & 27 & 237 & 0.27  & 27    & 348   & 3.7    & 6324  & 55     & 19    & 21     \\
    {tru9}  & 31 & 331 & 0.94  & 31    & 570   & 42     & 47579 & 468    & 20    & 179    \\
    {tru11} & 36 & 344 & 2.62  & 36    & 625   & 303    & \multicolumn{2}{c|}{maxit} & 22    & 2469  \\
    \bottomrule
    \end{tabular}%
  \label{tab:other}%
  }
\end{table}%

\subsection{No preconditioner}
Recall the two contributing factors to the efficiency of the presented methods: a) the dedicated preconditioner and b) the very fact of using an iterative method for the solution of linear systems, a method requiring only matrix-vector multiplications, without the need to assemble and store the actual matrix. While our main focus was on the preconditioner, the second factor is just as important. Indeed, it turns out that for several problem classes, no preconditioner is actually needed; typically for the matrix completion problems used in \cite{bellavia,Zhang_2017} or problems resulting form relaxation of combinatorial optimization problems. The typical feature of these problems is the lack of linear constraints and the fact that the matrix $\BAT\BA$ is diagonal. In such a case, our $H_\alpha$ is almost identical to $\widetilde{H}$ of \cite{Zhang_2017}. 

For example, when solving problem \verb|gpp500-1| from SDPLIB \cite{sdplib} (solution of rank~6), Loraine needed 288 total CG steps (2 in the last IP iteration) and 96 seconds to solve the problem using $H_\alpha$ and 4045 CG steps (539 in the last IP iteration) and (only) 14 seconds to solve the problem without preconditioner. 
In another example, we solved the matrix completion problem (\cite{bellavia,Zhang_2017}) with $n=36000$, $m=1200$ and $\rank X^*=1$; in this case Loraine needed 580 total CG steps and 53 seconds to solve the problem using $H_\alpha$ and 1108 CG steps and 48 seconds to solve the problem without preconditioner.

To demonstrate that this is \emph{not} the case of the truss topology problems, we include \cref{tab:noprecon}. It presents results of Loraine with no preconditioner for problems that are still solvable before reaching the maximum number of CG iteration in one linear system, set to 100000. We clearly see the advantage of using the preconditioner; confront~\cref{tab:tru,tab:vib}.
\begin{table}[htbp]
  \centering
  {\footnotesize\caption{Loraine with no preconditioner}
    \begin{tabular}{lrrr|lrrr}\toprule
    \multicolumn{1}{r}{} & \multicolumn{1}{r}{iter} & \multicolumn{1}{r}{CG iter} & \multicolumn{1}{r|}{CPU} &       & \multicolumn{1}{r}{iter} & \multicolumn{1}{r}{CG iter} & \multicolumn{1}{r}{CPU}  \\\midrule
    tru3  & 16    & 346   & 0.01    & {vib3} & 21    & 3115  & 0.06   \\
    tru5  & 21    & 2385  & 0.15    & {vib5} & 31    & 46949 & 2.76   \\
    tru7  & 27    & 10091 & 0.97   & {vib7} & 40    & 166135 & 24.9   \\
    tru9  & 30    & 33163 & 7.3    & {vib9} & 48    & 432416 & 181    \\
    tru11 & 36    & 77700 & 45     &       &       &       &         \\
    \bottomrule
    \end{tabular}%
  \label{tab:noprecon}%
  }
\end{table}%


\section{Conclusions}
We have introduced a preconditioner for the CG method that can be used in two different second-order algorithms for linear SDP. The preconditioner utilizes low rank of the SDP solutions. We have demonstrated that it is extremely efficient---as a preconditioner---for SDP problems arising from structural optimization.

How do the two optimization algorithms, PDAL and IP, compare? The interior-point algorithm needs, in general, less overall iterations to reach solution of the same quality. This confirms the observation from \cite{MKMS}. On the other hand, PDAL exhibits better conditioning of the system matrices when converging to the solution. This is particularly important when using iterative solvers; see \cref{fig:tr_IP_PDAL_cond}. So, while Loraine is ``faster'', in average, PDAL is capable of reaching higher accuracy of solutions, when required.

While we tested our software for other problems from standard libraries, we do not report it here for one of the following reasons:
\begin{enumerate}
    \item The available collections, such as \cite{sdplib}, are not representative enough in the sense that they do not include enough problems of the same type. We did not want to pick up single problems.
    \item Our algorithms benefit from the use of iterative solver for the linear system but no preconditioner is really needed. This is the case of matrix completion problems \cite{bellavia,Zhang_2017} or some problems arising in relaxation of combinatorial optimization.
    \item The problems do not satisfy our dimensional assumption $n\gg m$. For these problems the preconditioner is still very efficient (assuming low rank solutions) but the use of iterative methods does not bring any significant advantage to direct solvers. This is the case of many problems found in SDPLIB.
\end{enumerate}

Our next goal is to generalize the algorithms to SDP problems with low-rank solutions and free variables, with the goal to solve large-scale problems arising in polynomial optimization.






\bibliographystyle{siam}
\bibliography{bibliography}

\end{document}